\documentclass[12pt]{amsart}
\usepackage{amssymb, amsthm, tikz}
\usepackage{mathtools}
\usepackage{hyperref}
\usepackage{cleveref}
\usepackage{fullpage}
\usepackage{braket}
\usepackage{ytableau}

\newtheorem*{theorem*}{Theorem}
\newtheorem{theorem}{\textbf{Theorem}}[section]
\newtheorem{proposition}[theorem]{\textbf{Proposition}}
\newtheorem{corollary}[theorem]{\textbf{Corollary}}
\newtheorem{lemma}[theorem]{\textbf{Lemma}}

\theoremstyle{definition}

\newtheorem{remark}[theorem]{Remark}
\newtheorem{example}[theorem]{Example}

\newtheorem{assumption*}[theorem]{Assumption}

\numberwithin{equation}{section}

\DeclareMathOperator{\Z}{\mathbb{Z}}
\DeclareMathOperator{\C}{\mathbb{C}}

\DeclareMathOperator{\wt}{\operatorname{wt}}

\DeclareMathOperator{\sva}{\operatorname{a}}
\DeclareMathOperator{\svb}{\operatorname{b}}
\DeclareMathOperator{\svc}{\operatorname{c}}
\DeclareMathOperator{\svd}{\operatorname{d}}

\tikzstyle{path}=[line width=4pt, black, opacity=0.7]
\tikzset{every picture/.style={line width=1pt}}

\title{Free Fermionic Schur Functions}
\author{Slava Naprienko}
\address{Department of Mathematics\\
  University of North Carolina at Chapel Hill\\
  Chapel Hill, North Carolina, 27599, United States of America}
\email[Slava Naprienko]{slava@naprienko.com}

\begin{document}
	
\begin{abstract}
	We introduce a new family of Schur functions $s_{\lambda/\mu;a,b}(x/y)$ that depend on two sets of variables and two sequences of parameters. These \emph{free fermionic Schur functions} generalize and unify double, supersymmetric, and dual Schur functions from literature. These functions have a hidden symmetry that is manifested in the supersymmetric Cauchy identity
    \[
        \sum_{\lambda}s_{\lambda;a,b}(x/y)\widehat{s}_{\lambda;a,b}(z/w) = \prod_{i,j}\frac{1+y_iz_j}{1-x_iz_j}\frac{1+x_iw_j}{1-y_iw_j},
    \]
    where $\widehat{s}_{\lambda;a,b}(z/w) = s_{\lambda';b',a'}(w/z)$ are the dual functions. 
    
    Our approach is based on the integrable six vertex model with free fermionic Boltzmann weights. We show that these weights satisfy the \textit{refined Yang-Baxter equations}, which allows us to prove the generalizations of well-known properties of the Schur functions. We emphasize that some of our results and proofs are novel even in special cases.
\end{abstract}

\maketitle

\tableofcontents

\section{Introduction}
The \emph{Schur functions} $s_\lambda(x)$ are the symmetric functions that occur in various parts of mathematics. In representation theory, they are the characters of the polynomial irreducible representations of the general linear group; they are also the representatives of the characters of the irreducible representations of the symmetric group under the characteristic map. In algebraic geometry, they represent the Schubert cycles in the cohomology ring of Grassmannians. In mathematical physics, they are the polynomial tau functions of the KP hierarchy; they are also the image of the standard basis of the fermionic Fock space under the boson-fermion correspondence. And because of that, they are one of the most fundamental objects in algebraic combinatorics. Naturally, the Schur functions inspired multiple generalizations. 

The \textit{factorial Schur functions} $s_{\lambda}(x | a)$ is one such generalization. This family now depends on a sequence of parameters $(a_i)_{i \in \Z}$. In the special case $a_i = -i+1$, these functions were introduced in \cite{BL89,CL93} inspired by decomposition of tensor products of representations using particular bases. Another variation, with shifted variables, is known as \emph{shifted Schur functions} $s_\lambda^*(x)$ and was introduced and studied by Olshanski and Okounkov in \cite{OO97,OO98}. They form a natural basis for the center of the universal enveloping algebra $U(\mathfrak{gl}_n)$. Another variant, the \textit{double Schur functions} $s_{\lambda}(x \,||\, a)$, was introduced by Okounkov in \cite{Ok98} and developed by Molev in \cite{Mol09}. They differ from the factorial Schur functions only by a reparametrization. Knutson and Tao \cite{KT03} showed that double Schur functions correspond to Schubert cycles in the equivariant cohomology of the Grassmanian. Another application of the double Schur functions is the interpolation of the symmetric functions \cite{Ok98, MS99}.

The \emph{supersymmetric Schur functions} $s_\lambda(x/y)$ is a family of Schur functions coming from representation theory of Lie superalgebra of the general linear group $\mathfrak{gl}(m|n)$ \cite{Ser84, BR87}. For a comprehensive overview and relation to the representation theory of the superalgebra $\mathfrak{gl}(m|n)$, see \cite{M07} and \cite{CW12}. They also are a prominent object in algebraic combinatorics as a plethystic substitution of the classical Schur functions.

Molev \cite{Mol98} simultaneously generalized the factorial and supersymmetric Schur functions by introducing the \textit{factorial supersymmetric Schur functions} $s_\lambda(x/y \,||\, a)$. These functions generalize both the factorial and the supersymmetric Schur functions at the same time. However, Olshanski, Regev, and Vershik \cite{OlRV03} pointed out that Molev's generalization does not have the stability property. They introduced the \textit{Frobenius-Schur functions} $s_{\lambda; a}(x, y)$ as a shifted version of Molev's functions, which do possess the stability property. These functions enter the formula for the combinatorial dimension of a skew Young diagram in terms of the Frobenius coordinates, which plays a key role in the asymptotic character theory of the symmetric groups. When the numbers of variables $x$ and $y$ are equal, the Frobenius-Schur functions differ from Molev's functions only by a shift in the parameters. 

The \emph{dual Schur functions} $\widehat{s}_{\lambda/\mu}(x \,||\, a)$ is another important variation that was originally introduced by Molev in \cite{Mol09}. They are the dual functions to the double Schur functions in the sense of the duality of Hopf algebras; they have been used to study the dual Littlewood-Richardson polynomials, which occur in the comultiplication rules for the double Schur functions. In \cite{LLS21}, Lam, Lee, and Shimozono showed that the dual Schur functions represent the Schubert cycles in the equivariant homology of Grassmannians using the Hopf duality. 

In this paper, we introduce a new family of Schur functions $s_{\lambda/\mu; a,b}(x/y)$ which generalizes and unifies all of the examples above into one family of supersymmetric functions. They depend on a skew diagram $\lambda/\mu$, two sets of variables $x = (x_1,\dots,x_n)$ and $y = (y_1,\dots,y_n)$, and two doubly infinite sequences of parameters $a = (a_i)_{i \in \Z}$ and $b = (b_i)_{i \in \Z}$. These \emph{free fermionic Schur functions} $s_{\lambda/\mu; a,b}(x/y)$ have the following specializations:
\begin{enumerate}
    \item $s_{\lambda/\mu; 0,0}(x/0) = s_{\lambda/\mu}(x)$: classical Schur functions,
    \item $s_{\lambda/\mu;0,0}(x/y) = s_{\lambda/\mu}(x/y)$: supersymmetric Schur functions,
    \item $s_{\lambda/\mu;a,0}(x/a) = s_{\lambda/\mu}(x \,||\, a')$: double Schur functions,
    \item $s_{\lambda/\mu;a,0}(x/y) = s_{\lambda/\mu}(x/y \,||\, a')$: factorial supersymmetric Schur functions,
    \item $s_{\lambda/\mu; a,0}(x/y) = s_{\lambda/\mu; a}(x,y)$: Frobenius-Schur functions,
    \item $s_{\lambda;0,b}(x/0) = \widehat{s}_{\lambda}(x \,||\, b')$: dual Schur functions,
\end{enumerate}
where $a' = (-a_{-i+1})_{i \in \Z}$ is the dual sequence, and $a = 0$ means $a_i = 0$ for all $i \in \Z$. See \Cref{thm:degenerations} for details. 

The free fermionic Schur functions have the following remarkable duality that is manifested in the supersymmetric Cauchy identity: 
\[
	\sum_{\lambda}s_{\lambda; a,b}(x/y)\widehat{s}_{\lambda; a,b}(z/w) = \prod_{i}\frac{1+y_i z_j}{1-x_iz_j}\frac{1+x_i w_j}{1-y_i w_j},
\]
where $\widehat{s}_{\lambda/\mu; a,b}(x/y) = s_{\lambda'/\mu'; b', a'}(y/x)$ are the \emph{dual free fermionic Schur functions}. See \Cref{sec:cauchy} for details. This identity has several noteworthy properties:
\begin{enumerate}
	\item The right-hand side of the identity does not depend on two doubly infinite sequences of parameters $a$ and $b$.
	\item The dual functions use the symmetry in partitions $\lambda$, both sets of variables $x$ and $y$, and both sets of parameters $a$ and $b$. In other words, it is impossible to identify this duality without all of the data involved.
	\item Even in the special case $a = 0$ or $b = 0$, the supersymmetric Cauchy identity is a new result for the factorial supersymmetric Schur functions (Frobenius-Schur functions).
	\item To the best of our knowledge, it is the first generalization of the supersymmetric Cauchy identity beyond the result of Berele and Regev \cite{BR87}. 
\end{enumerate}

In addition to the Cauchy identity, we prove many results that are usually associated with the Schur functions. Among these properties are supersymmetry, combinatorial formulae, generating series for hook functions, the Jacobi-Trudi identity and variations, the Weyl determinant formula, the Berele-Regev factorization, and others. We emphasize that many of these results are novel even in special cases. 

Our approach is based on the integrable six vertex model with free fermionic Boltzmann weights. Specifically, we define the free fermionic Schur functions as the partition function of the six vertex model, where each vertex is assigned two row spectral parameters $x$ and $y$, and two column spectral parameters $a$ and $b$. The weights assigned to vertices of different types are defined as follows:
\[
    \begin{aligned}
        \sva_1(x,y; a, b) &= 1-b x,\\
        \sva_2(x,y; a,b) &= y + a,
    \end{aligned} \quad
    \begin{aligned}
        \svb_1(x,y; a,b) &= 1+b y,\\
        \svb_2(x,y; a,b) &= x - a,
    \end{aligned} \quad
    \begin{aligned}
        \svc_1(x,y; a,b) &= 1-ab,\\
        \svc_2(x,y; a,b) &= x+y.
    \end{aligned}
\]

The central tool in the theory of integrable lattice models is the Yang-Baxter equation. It relates the weights of two vertices by exchanging their row spectral parameters simultaneously. In particular, if $T(x,y; a,b)$ represents a vertex with labels $x,y$ and $a,b$, then the classical Yang-Baxter equation can be written as
\[
    R(x_1,y_1; x_2,y_2)T(x_1,y_1; a,b)T(x_2,y_2; a,b) = T(x_2,y_2; a,b)T(x_1,y_1; a,b)R(x_1,y_1; x_2,y_2).
\]
One of the main novelties of our work is the introduction of new \textit{refined Yang-Baxter equations}, which allow us to exchange the spectral parameters $x$ or $y$ separately:
\begin{align*}
        R^1(x_1,y_1;x_2,y_2)T(x_1,y_1;a,b)T(x_2,y_2; a,b) &= T(x_2,y_1; a,b)T(x_1, y_2; a,b)R^1(x_1,y_1;x_2,y_2),\\    R^2(x_1,y_1;x_2,y_2)T(x_1,y_2;a,b)T(x_2,y_2; a,b) &= T(x_1,y_2; a,b)T(x_2, y_1; a,b)R^2(x_1,y_1;x_2,y_2).
    \end{align*}

The refined Yang-Baxter equations lead to the refined Yang-Baxter algebra of the row transfer operators. This allows us to demonstrate that the free fermionic Schur functions defined in terms of our model are \emph{supersymmetric}, that is, symmetric separately in $x$ and $y$, and satisfy the cancellation property. We remark that it is impossible to get this result using the classical Yang-Baxter equation since it exchanges two parameters $x$ and $y$ at once. 

The free fermionic six vertex model has been previously used to define generalizations of Schur functions. In the pioneer work \cite{BBF11}, the classical Schur functions were given in terms of the free fermionic lattice model. In \cite{BMN14}, the result was extended to the factorial Schur functions. In \cite{Mot17ik}, a more general family of Schur functions was defined using a specific parametrization of weights. In \cite{ABPW21}, Aggarwal, Borodin, Petrov, and Wheeler studied the partition functions of the free fermionic six vertex models with a different parametrization of the weights and showed that the partition functions specialize to the factorial Schur functions and supersymmetric Schur functions. It is possible to relate the weights from \cite{ABPW21} to our weights through a sequence of reparametrizations, which means that many results are analogous to each other.

However, we remark that only with the help of the refined Yang-Baxter equations, we can give a natural proof that the resulting free fermionic Schur functions are supersymmetric. Moreover, with our choices of parametrization, normalization, and shift of parameters, these functions are stable under specialization of variables and thus can be seen as the functions in infinitely many variables as elements of the inverse limit of a graded ring of the supersymmetric functions. Furthermore, our parametrization unifies the factorial supersymmetric Schur functions and the dual Schur functions from Molev's work \cite{Mol09}. This unification is a novel and unexpected result. Even in the special case $a = 0$, the resulting functions are a novel family of the dual supersymmetric Schur functions which extend the approach developed by Molev. By unifying these various types of Schur functions and establishing their stability and supersymmetry, we provide a uniform approach to study all of these functions using the free fermionic six vertex model as our main tool.

\bigskip

\textbf{Acknowledgements.} 
I would like to express my sincere gratitude to Daniel Bump for his invaluable support, guidance, and mentorship throughout the course of this project. I also thank Grigori Olshanski for suggesting the mentioning the dual Cauchy identity. I also thank Anne Schilling and Leonid Petrov for helpful discussions. I am also deeply grateful to the referee for their careful reading and numerous helpful suggestions.

\section{Preliminaries}
We begin by reviewing some standard notation and terminology related to partitions, following Chapter I of \cite{Mac95}. We also introduce the double shifted powers.

A partition $\lambda = (\lambda_1,\lambda_2,\dots,\lambda_n)$ is a non-increasing sequence of non-negative integers. We use Frobenius notation to express partitions as $\lambda = (p | q) = (p_1, p_2, \dots, p_d | q_1, q_2, \dots, q_d)$, where $p_i = \lambda_i - i$ and $q_i = \lambda_i' - i$, and $d$ is the number of boxes in the diagonal of the corresponding Young diagram. The conjugate partition of $\lambda$ is denoted by $\lambda'$, and in Frobenius notation, it is given by $\lambda' = (q | p)$.

Let $\lambda/\mu$ be a skew diagram obtained by removing the boxes of the partition $\mu$ from the partition $\lambda$. A box $\alpha = (i,j) \in \lambda/\mu$ has a left neighbor if $(i,j-1) \in \lambda/\mu$, a right neighbor if $(i,j+1) \in \lambda/\mu$, a top neighbor if $(i-1,j) \in \lambda/\mu$, and a bottom neighbor if $(i+1,j) \in \lambda/\mu$. The content $c(\alpha) = j-i$ of a box $\alpha = (i,j)$ in $\lambda/\mu$ is the difference between its row and column indices.

A connected skew diagram $\lambda/\mu$ that contains no $2 \times 2$ blocks of squares is called a \emph{ribbon} or a \emph{skew hook}. Any skew diagram with no $2 \times 2$ blocks of squares is a disjoint union of ribbons. In a ribbon, every box can have either a left or bottom neighbor, but not both. A box in a ribbon can have neither a left nor a bottom neighbor if it is the southwestmost box. In a union of ribbons, every box either has a right or a top neighbor, but not both. A box in a ribbon can have neither a right nor a top neighbor if it is the northeastmost box. Also, in a ribbon, every box has distinct content.

Let $R(\lambda/\mu)$ be the set of ribbons in a skew diagram $\lambda/\mu$ that has no $2 \times 2$ blocks of boxes. 

Let $a = (a_i)_{i \in \Z}$ be a sequence of indeterminates or complex numbers. We write $a' = (-a_{-i+1})_{i \in \Z}$ for the dual sequence. Note that the operation defining $a'$ is an involution. We also often write $a = 0$ to mean $a_i = 0$ for all $i \in \Z$. Let $x = (x_1,\dots,x_n)$ be a finite sequence of indeterminates. Then, when we write $x = a$, we mean that $x_i = a_i$ for all $1 \leq i \leq n$.

We define the shifted powers $(x | a)^0 = 1$, $(x ; a)^0 = 1$, and for $k > 0$, we define $(x | a)^k$ and $(x;a)^k$ as follows:
\begin{align}
	(x|a)^k &= (x-a_1)(x-a_2)\dots (x-a_k),\\
	(x|a)^{-k} &= \frac{1}{(x-a_0)(x-a_{-1})\dots (x-a_{-k+1})},\\
	(x;a)^k &= (1-a_1 x)(1-a_2 x)\dots (1-a_k x),\\
	(x;a)^{-k} &= \frac{1}{(1-a_0x)(1-a_{-1}x)\dots (1-a_{-k+1}x)}.
\end{align}

Let $a = (a_i)_{i \in \Z}$ and $b = (b_i)_{i \in \Z}$ be two sequences of indeterminates or complex numbers. We define the double shifted powers $(x | a,b)^k$ by
\begin{equation}
	(x | a,b)^k = \frac{(x|a)^k}{(x;b)^k}, \quad k \in \Z.
\end{equation}

Note that $(x|a,b)^{-k} = 1/(x | -a',-b')^k$ and $(x^{-1} | a, b)^k = (x | -b', -a')^{-k}$. 

\section{The six vertex model}

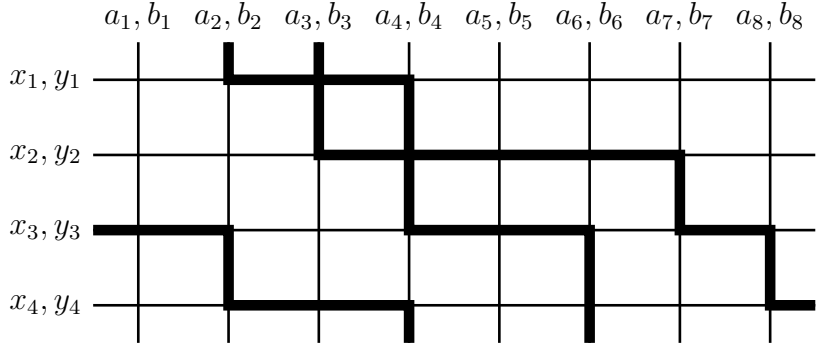
\begin{figure}
    \centering
    \begin{tikzpicture}[xscale=0.6, yscale=0.5]
        \foreach \x in {1,3,5,7} {
            \draw (0,\x) -- (16,\x);
            \pgfmathtruncatemacro{\index}{5-(\x+1)/2}
            \node [left] at (0,\x) {$x_\index, y_\index$};
        }
        \foreach \y in {1,3,5,7,9,11,13,15} {
            \draw (\y,0) -- (\y,8);
            \pgfmathtruncatemacro{\index}{(\y+1)/2}
            \node [above] at (\y,8) {$a_\index, b_\index$};
        }
        
        \draw[path] (0,3) -- (1,3) -- (3,3) -- (3,1) -- (7,1) -- (7,0);
        
        \draw[path] (3,8) -- (3,7) -- (7,7) -- (7,3) -- (11,3) -- (11,1) -- (11,0);
        
        \draw[path] (5,8) -- (5,5) -- (13,5) -- (13,3) -- (15,3) -- (15,1) -- (16, 1);
    \end{tikzpicture}
    \caption{A typical state in the six vertex model.}
    \label{fig:svmodel}
\end{figure}

\begin{figure}
    \centering
    \[
    \begin{array}{cccccc}
        \begin{tikzpicture}[scale=0.75]
            \draw (0,0) -- (2,0);
            \draw (1,-1) -- (1,1);
        \end{tikzpicture} &
        \begin{tikzpicture}[scale=0.75]
            \draw (0,0) -- (2,0);
            \draw (1,-1) -- (1,1);
            \draw[path] (0,0) -- (2,0);
            \draw[path] (1,-1) -- (1,1);
        \end{tikzpicture} &
        \begin{tikzpicture}[scale=0.75]
            \draw (0,0) -- (2,0);
            \draw (1,-1) -- (1,1);
            \draw[path] (1,-1) -- (1,1);
        \end{tikzpicture} &
        \begin{tikzpicture}[scale=0.75]
            \draw (0,0) -- (2,0);
            \draw (1,-1) -- (1,1);
            \draw[path] (0,0) -- (2,0);
        \end{tikzpicture} &
        \begin{tikzpicture}[scale=0.75]
            \draw (0,0) -- (2,0);
            \draw (1,-1) -- (1,1);
            \draw[path] (1,1) -- (1,0) -- (2,0);
        \end{tikzpicture} &
        \begin{tikzpicture}[scale=0.75]
            \draw (0,0) -- (2,0);
            \draw (1,-1) -- (1,1);
            \draw[path] (0,0) -- (1,0) -- (1,-1);
        \end{tikzpicture} \\
        \sva_1 & \sva_2 & \svb_1 & \svb_2 & \svc_1 &  \svc_2
    \end{array}
    \]
    \caption{The six admissible types of vertices. The types of vertices are traditionally called $\sva_1,\sva_2,\svb_1,\svb_2,\svc_1,\svc_2$, following Baxter \cite{Bax82}}
    \label{fig:sixtypes}
\end{figure}
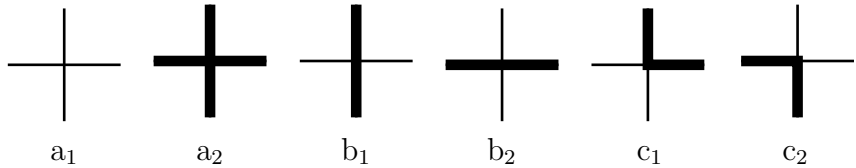

Let us review the main object of the paper -- the six vertex model. 

The six vertex model is a combinatorial system on a rectangular grid like the one in \Cref{fig:svmodel}. The grid consists of vertices and edges. The edges can be of two types: the internal edges which are adjacent to two vertices, and external edges which are adjacent only to one vertex. Each edge can be either occupied, in which case we draw it by a thick line, or unoccupied, in which case we draw it by thin line. Each edge can carry at most one path. In the six vertex model, we allow only the six types of vertices out of sixteen to occur, they are shown in \Cref{fig:sixtypes}. Because of this restriction, the six vertex model bears its name. It is convenient to think of the six vertex model as a lattice path system, where paths enter the grid from the left and from the top, leave the grid on the right and bottom. Note that the paths can intersect, but they can move only right and down. Because of the vertices of type $\sva_2$, the paths in the model can intersect.

An (admissible) \emph{state} in the six vertex model is an arrangement of occupied and unoccupied edges such that only the allowed six types of vertices occur. See \Cref{fig:svmodel} for an example of a state. The word ``admissible'' is a linguistic convenience to separate the valid states where only allowed six vertices happen from the invalid states. Typically, we fix the boundary edges, and consider all possible states of the internal edges. With fixed boundaries, there are only finitely admissible states, which gives a well-defined combinatorial device. 

Usually, the rows and columns of six vertex model are decorated with so-called \emph{spectral parameters} which are some independent indeterminates or complex numbers. In this paper, we decorate each row of the six vertex model with a pair $(x,y)$ and each column with a pair $(a,b)$. We will always use letters $x,y,z,w$ for the row parameters, and letters $a,b$ for the column parameters. 

Now we assign weight to each admissible state. The weight $\wt(s)$ of a state  $s$ is the product of weights of its vertices. Now, to define the weight $\wt(v)$ of a vertex $v$, we equip ourselves with the following six vertex weight functions which depends on the spectral parameters of the row and the column where the vertex occurred.
\begin{equation}\label{eq:ffweights}
    \begin{aligned}
        \sva_1(x,y; a, b) &= 1-b x,\\
        \sva_2(x,y; a,b) &= y + a,
    \end{aligned} \quad
    \begin{aligned}
        \svb_1(x,y; a,b) &= 1+b y,\\
        \svb_2(x,y; a,b) &= x - a,
    \end{aligned} \quad
    \begin{aligned}
        \svc_1(x,y; a,b) &= 1-ab,\\
        \svc_2(x,y; a,b) &= x+y.
    \end{aligned}
\end{equation}
Then a vertex $v$ of type $\sva_1$ that occurred on a row and column with spectral parameters $(x,y)$ and $(a,b)$, respectively, has the weight $\wt(v) = \sva_1(x,y; a,b)$, and so on. 

Finally, for a six vertex model with fixed boundaries, the \emph{partition function} is the sum of weights of all admissible states respecting the boundaries:
\[
	Z = \sum_{s}\wt(s) = \sum_{s}\prod_{v \in s}\wt(v),
\]
where $s$ are the admissible states in the model and $v$ are the vertices in state $s$. 

From the combinatorial point of view, one of the main objectives in the study of integrable lattice models is to identify appropriate weights that result in meaningful and useful partition functions. The six vertex model, for example, has been shown to produce a variety of special functions depending on the choice of weights used. With one set of weights, the model generates the number of alternating sign matrices \cite{Kup96}. With another, it produces Schur functions, and more generally, spherical Whittaker functions for the general linear group over a non-archimedean local field by means of the Casselman-Shalika formula \cite{BBF11}. Using the weights which depends on both rows and columns, one can get the factorial Schur functions \cite{BMN14}. Additionally, using yet another set of weights, the six vertex model generates supersymmetric Schur functions \cite{Har21} and is related to Hamiltonian operators. Recently, by using more general weights, it has been demonstrated that the six vertex model can produce various generalizations of Schur functions \cite{Mot17ik,Mot17,ABPW21}. See also \cite{ZJ09} for the relations of the six vertex model with other combinatorial objects.

A six vertex model is called \emph{free fermionic} if the weights $\sva_1,\sva_2,\svb_1,\svb_2,\svc_1,\svc_2$ satisfy 
\begin{equation}
	\sva_1 \sva_2 + \svb_1 \svb_2 = \svc_1 \svc_2
\end{equation}
for all vertices in the model. The weights \eqref{eq:ffweights} are free fermionic because we have
	\begin{equation}
		(1-b x)(y+a) + (1+by)(x-a) = (1-ab)(x+y).
	\end{equation}
Free fermionic weights provide the six vertex model with two extra features. First, such six vertex model can be interpreted in terms of the non-intersecting lattice paths. We develop this theory in \Cref{sec:lgvforsv}. Second, the free fermionic weights have a larger space of solutions of the Yang-Baxter equation. See \cite{BBF11, N22} for details. 

\subsection{Row transfer operators}
One poweful powerful way to study the combinatorics of the six vertex model is by means of liner algebra. It gives yet another example of the common belief that \emph{math is the art of reducing problems to linear algebra}. 

Let us consider one vertex with row parameters $(x,y)$ and column parameters $(a,b)$:
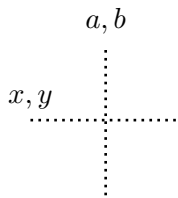
\[
	\begin{tikzpicture}
        \draw[dotted] (0,0) -- (2,0);
        \draw[dotted] (1,-1) -- (1,1);
        \node [above] at (0,0) {\small$x,y$};
        \node [above] at (1,1) {\small$a,b$};
    \end{tikzpicture}
\]

We associate with the row a two-dimensional vector space $V(x,y)$ and with the column a two dimensional vector space $W(a,b)$. Both vector spaces have a distinguished basis $e_0, e_1$, where $e_0$ corresponds to the unoccupied edge, and $e_1$ corresponds to an occupied edge. 

Then we can see the vertex as an operator 
\[
	T(x,y; a,b)\colon V(x,y) \otimes W(a,b) \to W(a,b) \otimes V(x,y)
\]
given explicitly in the standard basis $e_0 \otimes e_0, e_0 \otimes e_1, e_1 \otimes e_0, e_1 \otimes e_1$ by the matrix
\[
	\begin{pmatrix}
        \sva_1(x,y; a,b) & & & \\
        & \svc_1(x,y;a,b)& \svb_1(x,y; a,b) &\\
        & \svb_2(x,y; a,b) & \svc_2(x,y; a,b) & \\
        & & & \sva_2(x,y; a,b)
    \end{pmatrix} = \begin{pmatrix}
        1-b x & & & \\
        & 1-a b & 1+b y &\\
        & x-a & x+y & \\
        & & & y+a
    \end{pmatrix}.
\]
Now we can interpret the weights of the admissible vertices as the matrix elements of the operator $T(x,y; a,b)$. For example,
\[
	\begin{tikzpicture}[baseline={([yshift=-2ex]current bounding box.center)}]
            \draw (0,0) -- (2,0);
            \draw (1,-1) -- (1,1);
            \draw[path] (1,-1) -- (1,1);
        \node [above] at (0,0) {\small$V(x,y)$};
        \node [above] at (1,1) {\small$W(a,b)$};
        \end{tikzpicture}= \braket{e_0 \otimes e_1 | T(x,y; a,b) | e_1 \otimes e_0} = \svb_1(x,y; a,b) = 1+by.
\]
Here and everywhere in text, we use the braket notation $\braket{v | T | w}$ to represent the matrix coefficient of the operator $T$, vector $v$ and the dual vector $w$. We always assume that we chose the dual basis corresponding to the distinguished basis $e_0,e_1$ of all two-dimensional vector spaces that occur.

Now that we interpret a vertex as an operator, we will consider an entire row in the six vertex model with row parameters $(x,y)$ and column parameters $(a_1,b_1), \dots, (a_m, b_m)$:
\[
   \begin{tikzpicture}[dotted, scale=0.8]
        \draw (0,0) -- (16,0);
        \draw (1,-1) -- (1,1);
        \draw (3,-1) -- (3,1);
        \draw (5,-1) -- (5,1);
        \draw (7,-1) -- (7,1);
        \draw (9,-1) -- (9,1);
        \draw (11,-1) -- (11,1);
        \draw (13,-1) -- (13,1);
        \draw (15,-1) -- (15,1);
        \node [above] at (0,0) {\small$x,y$};
        \node [above] at (1,1) {\small$a_1,b_1$};
        \node [above] at (3,1) {\small$a_2,b_2$};
        \node [above] at (5,1) {\small$a_3,b_3$};
        \node [above] at (7,1) {\small$a_4,b_4$};
        \node [above] at (9,1) {\small$a_5,b_5$};
        \node [above] at (11,1) {\small$a_6,b_6$};
        \node [above] at (13,1) {\small$a_7,b_7$};
        \node [above] at (15,1) {\small$a_8,b_8$};
    \end{tikzpicture} 
\]

We again associate a vector space $V(x,y)$ to the row with spectral parameters $(x,y)$, and to each column with spectral parameters $(a_k, b_k)$, we associate the vector space $W(a_k,b_k)$. Let $a = (a_1,\dots,a_m)$ and $b = (b_1,\dots,b_m)$. Then we associate a vector space $W(a,b)$ with the entire space of columns by $W(a,b) = W(a_1,b_1) \otimes W(a_2,b_2) \otimes \dots \otimes W(a_m,b_m)$.  Then the row in the six vertex model can be represented by an operator
\[
	T(x,y; a,b) = T(x,y; a_1, b_1)T(x,y; a_2, b_2)\dots T(x,y; a_m, b_m),
\]
from $V(x,y) \otimes W(a,b)$ to $W(a,b) \otimes V(x, y)$. In this notation, we assume that the operator $T(x,y; a_k, b_k)$ acts on the parametrized spaces $V(x,y) \otimes W(a_k, b_k)$ by itself, and by identity everywhere else. For example, the weight of the row
\[
   \begin{tikzpicture}[scale=0.8]
        \draw (0,0) -- (12,0);
        \draw (1,-1) -- (1,1);
        \draw (3,-1) -- (3,1);
        \draw (5,-1) -- (5,1);
        \draw (7,-1) -- (7,1);
        \draw (9,-1) -- (9,1);
        \draw (11,-1) -- (11,1);
        \node [above] at (0,0) {\small$x,y$};
        \node [above] at (1,1) {\small$a_1,b_1$};
        \node [above] at (3,1) {\small$a_2,b_2$};
        \node [above] at (5,1) {\small$a_3,b_3$};
        \node [above] at (7,1) {\small$a_4,b_4$};
        \node [above] at (9,1) {\small$a_5,b_5$};
        \node [above] at (11,1) {\small$a_6,b_6$};
        
        \draw[path] (0,0) -- (1,0) -- (1, -1);
        \draw[path] (5,1) -- (5,0) -- (7, 0) -- (7, -1);
        \draw[path] (9, 1) -- (9,0) -- (12, 0);
    \end{tikzpicture} 
\]
is equal to 
\begin{multline*}
	\braket{e_1 \otimes (e_0 \otimes e_0 \otimes e_1 \otimes e_0 \otimes e_1 \otimes e_0) | T(x,y; a,b) | (e_1 \otimes e_0 \otimes e_0 \otimes e_1 \otimes e_0 \otimes e_0) \otimes e_1} = \\
	= \svc_2(x,y; a_1,b_1)\sva_1(x,y; a_2,b_2)\svc_1(x,y; a_3,b_3)\svc_2(x,y; a_4,b_4)\svc_1(x,y; a_5, b_5)\svb_2(x,y; a_6, b_6).
\end{multline*}

Of course, by stacking together the row operators with different row spectral parameters $(x_1,y_1), \dots, (x_n, y_n)$, we get the operator representing the entire six vertex model. However, we will not need the operator for the entire model, so we omit the details. 

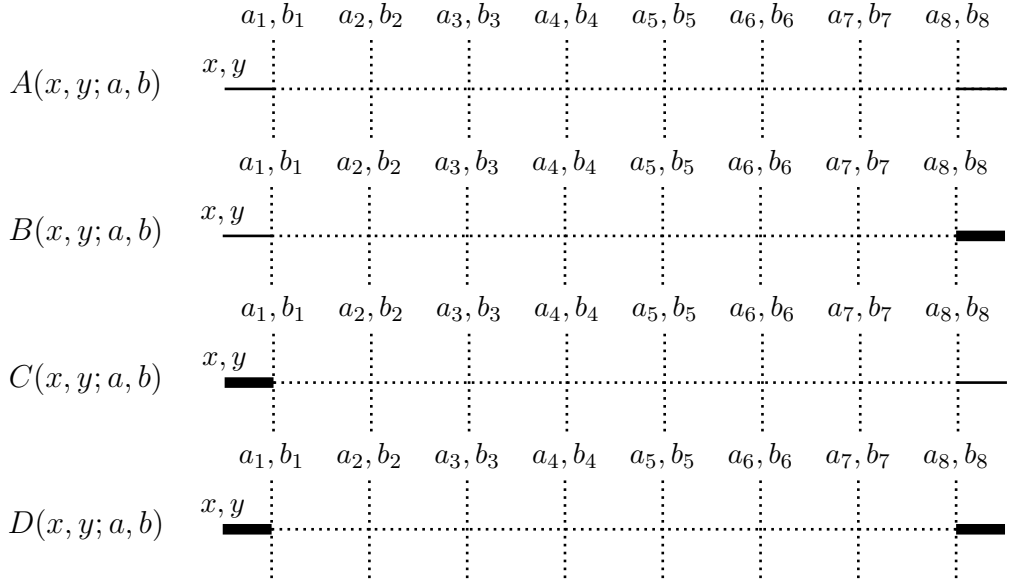
\begin{figure}
    \centering
    \[
        \begin{array}{cc}
        A(x,y; a,b) & \begin{tikzpicture}[baseline={([yshift=-2ex]current bounding box.center)}, scale=0.65]
        \draw[dotted] (1,0) -- (16,0);
        \draw[dotted] (1,-1) -- (1,1);
        \draw[dotted] (3,-1) -- (3,1);
        \draw[dotted] (5,-1) -- (5,1);
        \draw[dotted] (7,-1) -- (7,1);
        \draw[dotted] (9,-1) -- (9,1);
        \draw[dotted] (11,-1) -- (11,1);
        \draw[dotted] (13,-1) -- (13,1);
        \draw[dotted] (15,-1) -- (15,1);
        \draw (0, 0) -- (1, 0);
        \draw (15, 0) -- (16, 0);
        \node [above] at (0,0) {\small$x,y$};
        \node [above] at (1,1) {\small$a_1,b_1$};
        \node [above] at (3,1) {\small$a_2,b_2$};
        \node [above] at (5,1) {\small$a_3,b_3$};
        \node [above] at (7,1) {\small$a_4,b_4$};
        \node [above] at (9,1) {\small$a_5,b_5$};
        \node [above] at (11,1) {\small$a_6,b_6$};
        \node [above] at (13,1) {\small$a_7,b_7$};
        \node [above] at (15,1) {\small$a_8,b_8$};
    \end{tikzpicture}  \\
        B(x,y; a,b) & \begin{tikzpicture}[baseline={([yshift=-2ex]current bounding box.center)}, scale=0.65]
        \draw[dotted] (1,0) -- (15,0);
        \draw[dotted] (1,-1) -- (1,1);
        \draw[dotted] (3,-1) -- (3,1);
        \draw[dotted] (5,-1) -- (5,1);
        \draw[dotted] (7,-1) -- (7,1);
        \draw[dotted] (9,-1) -- (9,1);
        \draw[dotted] (11,-1) -- (11,1);
        \draw[dotted] (13,-1) -- (13,1);
        \draw[dotted] (15,-1) -- (15,1);
        \draw (0, 0) -- (1, 0);
        \draw[path] (15,0) -- (16,0);
        \node [above] at (0,0) {\small$x,y$};
        \node [above] at (1,1) {\small$a_1,b_1$};
        \node [above] at (3,1) {\small$a_2,b_2$};
        \node [above] at (5,1) {\small$a_3,b_3$};
        \node [above] at (7,1) {\small$a_4,b_4$};
        \node [above] at (9,1) {\small$a_5,b_5$};
        \node [above] at (11,1) {\small$a_6,b_6$};
        \node [above] at (13,1) {\small$a_7,b_7$};
        \node [above] at (15,1) {\small$a_8,b_8$};
    \end{tikzpicture} \\
        C(x,y; a,b) & \begin{tikzpicture}[baseline={([yshift=-2ex]current bounding box.center)}, scale=0.65]
        \draw[dotted] (1,0) -- (15,0);
        \draw[dotted] (1,-1) -- (1,1);
        \draw[dotted] (3,-1) -- (3,1);
        \draw[dotted] (5,-1) -- (5,1);
        \draw[dotted] (7,-1) -- (7,1);
        \draw[dotted] (9,-1) -- (9,1);
        \draw[dotted] (11,-1) -- (11,1);
        \draw[dotted] (13,-1) -- (13,1);
        \draw[dotted] (15,-1) -- (15,1);
        \draw[path] (0,0) -- (1,0);
        \draw (15,0) -- (16,0);
        \node [above] at (0,0) {\small$x,y$};
        \node [above] at (1,1) {\small$a_1,b_1$};
        \node [above] at (3,1) {\small$a_2,b_2$};
        \node [above] at (5,1) {\small$a_3,b_3$};
        \node [above] at (7,1) {\small$a_4,b_4$};
        \node [above] at (9,1) {\small$a_5,b_5$};
        \node [above] at (11,1) {\small$a_6,b_6$};
        \node [above] at (13,1) {\small$a_7,b_7$};
        \node [above] at (15,1) {\small$a_8,b_8$};
    \end{tikzpicture} \\
        D(x,y; a,b) & \begin{tikzpicture}[baseline={([yshift=-2ex]current bounding box.center)}, scale=0.65]
        \draw[dotted] (1,0) -- (15,0);
        \draw[dotted] (1,-1) -- (1,1);
        \draw[dotted] (3,-1) -- (3,1);
        \draw[dotted] (5,-1) -- (5,1);
        \draw[dotted] (7,-1) -- (7,1);
        \draw[dotted] (9,-1) -- (9,1);
        \draw[dotted] (11,-1) -- (11,1);
        \draw[dotted] (13,-1) -- (13,1);
        \draw[dotted] (15,-1) -- (15,1);
        \draw[path] (0,0) -- (1,0);
        \draw[path] (15,0) -- (16,0);
        \node [above] at (0,0) {\small$x,y$};
        \node [above] at (1,1) {\small$a_1,b_1$};
        \node [above] at (3,1) {\small$a_2,b_2$};
        \node [above] at (5,1) {\small$a_3,b_3$};
        \node [above] at (7,1) {\small$a_4,b_4$};
        \node [above] at (9,1) {\small$a_5,b_5$};
        \node [above] at (11,1) {\small$a_6,b_6$};
        \node [above] at (13,1) {\small$a_7,b_7$};
        \node [above] at (15,1) {\small$a_8,b_8$};
    \end{tikzpicture}
    \end{array}
    \]
    \caption{The row transfer operators.}
    \label{fig:rowtransferoperators}
\end{figure}

It turns out to be convenient to define the row transfer operators
\[
	A(x,y), B(x,y), C(x,y), D(x,y)\colon W(a; b) \to W(a; b),
\]
 which are the operators on the space of the columns when the boundaries for the row boundary edges are fixed. See \Cref{fig:rowtransferoperators} for illustration. In terms of the row operator $T(x,y; a,b)$, the row transfer operators can be defined by their matrix coefficients:
\begin{align*}
    \langle v | A(x,y) | w \rangle &= \langle e_0 \otimes v | T(x,y; a,b) | w \otimes e_0 \rangle,\\
    \langle v | B(x,y) | w \rangle &= \langle e_1 \otimes v | T(x,y;a,b) | w \otimes e_0 \rangle,\\
    \langle v | C(x,y) | w \rangle &= \langle e_1 \otimes v | T(x,y;a,b) | w \otimes e_0 \rangle,\\
    \langle v | D(x,y) | w \rangle &= \langle e_1 \otimes v | T(x,y; a,b) | w \otimes e_1 \rangle.
\end{align*}

Note that we suppress the dependence on $a,b$ from notation because the operators are defined equally well for one or multiple columns with any column parameters. 

\begin{lemma}\label{lem:vanishing}
	We have $\braket{v | A(t,-t) | w} = 0$ unless $v = w$.
\end{lemma}
\begin{proof}
	When $x = y$, the vertices of type $\svc_2$ have weight $0$, so other states can't occur.
\end{proof}

It is convenient to introduce a ``swap operator'' $P$ defined by $P(v \otimes w) = w \otimes v$ and extended by linearity. Then we define the monodromy matrix $M(x,y; a,b) = T(x,y; a,b)P$ when $a,b$ are single parameters, and 
\[
	M(x,y; a,b) = M(x,y; a_1,b_1)\dots M(x,y; a_m, b_m),
\]
when $a = (a_1,\dots,a_m)$ and $b=(b_1,\dots,b_m)$. Then we have
\begin{equation}
	M(x,y) = \begin{pmatrix}
		A(x,y) & B(x,y) \\
		C(x,y) & D(x,y)
	\end{pmatrix},
\end{equation}
where we again suppressed the dependence on the column parameters. 

The monodromy matrix is a convenient matrix to keep track of all four row transfer operators at the same time. In \Cref{sec:infinitesv}, we show that the monodromy matrix is more suitable for the extension to the case of infinitely many columns than the operator $T(x,y;a,b)$.

When the six vertex model is represented by the row transfer operators, the partition functions become the matrix coefficients of the row transfer operators on the standard basis elements. For example, the partition function of the six vertex model with boundaries from \Cref{fig:svmodel} can be represented in the operator form as the matrix coefficient
\[
	\braket{v | A(x_1,y_1)A(x_2,y_2)C(x_3,y_3)B(x_4, y_4) | w },
\]
where $v = e_0 \otimes e_1 \otimes e_1 \otimes e_0 \otimes e_0 \otimes e_0 \otimes e_0 \otimes e_0$ and $w = e_0 \otimes e_0 \otimes e_0 \otimes e_1 \otimes e_0 \otimes e_1 \otimes e_0 \otimes e_0$.

While the language of the row transfer operators is convenient for the analysis of the partition functions and writing functional relations, the combinatorial nature of the six vertex model is somewhat lost behind the language of the operators. Some of the properties of the six vertex model can be better understood in terms of the operators, and some -- in terms of the combinatorics of the lattice model, see \Cref{sec:lgvforsv} for an example. 

\subsection{Dual six vertex model}\label{sec:dualsv}
In addition to the standard six vertex model, where the paths travel down and right, we need a dual version of the model where paths travel down and left. See \Cref{fig:dualmodel} for illustration. The regular six vertex model and the dual version correspond to the gamma-ice and delta-ice from \cite{BBF11}. 

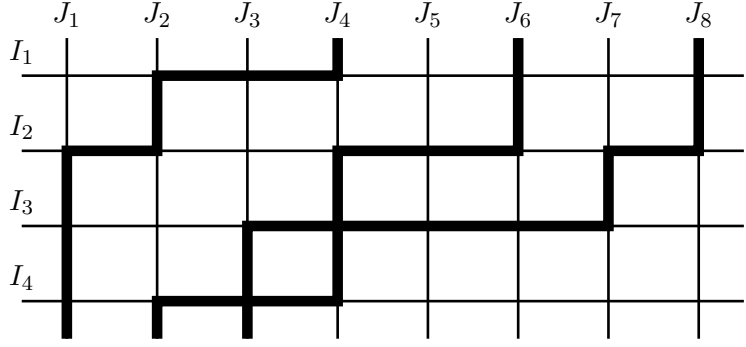
\begin{figure}
    \centering
    \begin{tikzpicture}[xscale=0.6, yscale=0.5]
        \foreach \x in {1,3,5,7} {
            \draw (0,\x) -- (16,\x);
            \pgfmathtruncatemacro{\index}{5-(\x+1)/2}
            \node [above] at (0,\x) {\small $I_\index$};
        }
        \foreach \y in {1,3,5,7,9,11,13,15} {
            \draw (\y,0) -- (\y,8);
            \pgfmathtruncatemacro{\index}{(\y+1)/2}
            \node [above] at (\y,8) {\small $J_\index$};
        }
        
        \draw[path] (1,0) -- (1,5) -- (3,5) -- (3,7) -- (7,7) -- (7,8);
        
        \draw[path] (3,0) -- (3,1) -- (7,1) -- (7,5) -- (11,5) -- (11,7) -- (11,8);
        
        \draw[path] (5,0) -- (5,3) -- (13,3) -- (13,5) -- (15,5) -- (15,8);
    \end{tikzpicture}
    \caption{A typical state in the dual six vertex model.}
    \label{fig:dualmodel}
\end{figure}

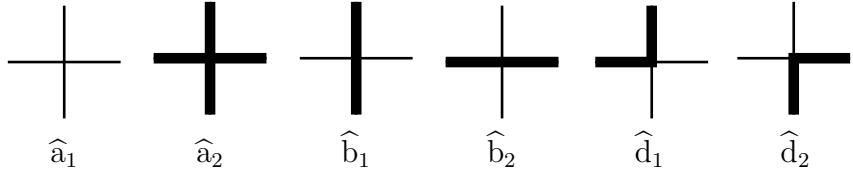
\begin{figure}
    \centering
    \[
    \begin{array}{cccccc}
        \begin{tikzpicture}[scale=0.75]
            \draw (0,0) -- (2,0);
            \draw (1,-1) -- (1,1);
        \end{tikzpicture} &
        \begin{tikzpicture}[scale=0.75]
            \draw (0,0) -- (2,0);
            \draw (1,-1) -- (1,1);
            \draw[path] (0,0) -- (2,0);
            \draw[path] (1,-1) -- (1,1);
        \end{tikzpicture} &
        \begin{tikzpicture}[scale=0.75]
            \draw (0,0) -- (2,0);
            \draw (1,-1) -- (1,1);
            \draw[path] (1,-1) -- (1,1);
        \end{tikzpicture} &
        \begin{tikzpicture}[scale=0.75]
            \draw (0,0) -- (2,0);
            \draw (1,-1) -- (1,1);
            \draw[path] (0,0) -- (2,0);
        \end{tikzpicture} &
        \begin{tikzpicture}[scale=0.75]
            \draw (0,0) -- (2,0);
            \draw (1,-1) -- (1,1);
            \draw[path] (1,1) -- (1,0) -- (0,0);
        \end{tikzpicture} &
        \begin{tikzpicture}[scale=0.75]
            \draw (0,0) -- (2,0);
            \draw (1,-1) -- (1,1);
            \draw[path] (2,0) -- (1,0) -- (1,-1);
        \end{tikzpicture} \\
        \widehat{\sva}_1 & \widehat{\sva}_2 & \widehat{\svb}_1 & \widehat{\svb}_2 & \widehat{\svd}_1 & \widehat{\svd}_2
    \end{array}
    \]
    \caption{The six admissible types of vertices of the dual six vertex model. The types of vertices are traditionally called $\widehat{\sva}_1,\widehat{\sva}_2,\widehat{\svb}_1,\widehat{\svb}_2,\widehat{\svd}_1,\widehat{\svd}_2$, following the special case of the eight vertex model from \cite{Bax82}}
    \label{fig:dualsixtypes}
\end{figure}

Because of the change of the direction of paths, the six admissible types of vertices change such that $\svc_1$ and $\svc_2$ turn into the types $\widehat{\svd}_1$ and $\widehat{\svd}_2$. To make the difference between the vertices at all times, when we discuss the dual six vertex model, we ornament all types with the hats. In, particular, the six types of the dual six vertex models are shown in \Cref{fig:dualsixtypes}. We again decorate rows and columns with spectral parameters and define the dual vertex weight functions, which are different from the standard weights \eqref{eq:ffweights} only in swapping parameters $a,b$ for the corresponding types of vertices:
\begin{equation}\label{eq:dualffweights}
    \begin{aligned}
        \widehat{\sva}_1(x,y; a, b) &= 1-a x,\\
        \widehat{\sva}_2(x,y; a,b) &= y + b,
    \end{aligned} \quad
    \begin{aligned}
        \widehat{\svb}_1(x,y; a,b) &= 1+a y,\\
        \widehat{\svb}_2(x,y; a,b) &= x - b,
    \end{aligned} \quad
    \begin{aligned}
        {\widehat{\svd}}_1(x,y; a,b) &= 1-ab,\\
        {\widehat{\svd}}_2(x,y; a,b) &= x+y.
    \end{aligned}
\end{equation}

We also represent vertices of the dual six vertex model in terms of operators acting on the vector spaces $V(x,y)$ and $W(a,b)$. The operator representing a vertex is $\widehat{T}(x,y; a,b)\colon V(x,y) \otimes W(a,b) \to W(a,b) \otimes V(x,y)$ is given in the standard basis by
\[
	\begin{pmatrix}
        \widehat{\sva}_1(x,y; a,b) & & & \widehat{\svd}_2(x,y; a,b)\\
        &  & \widehat{\svb}_1(x,y; a,b) &\\
        & \widehat{\svb}_2(x,y; a,b) &  & \\
        \widehat{\svd}_1(x,y; a,b) & & & \widehat{\sva}_2(x,y; a,b)
    \end{pmatrix} = \begin{pmatrix}
        1-a x & & & x+y\\
        & & 1+a y &\\
        & x-b &  & \\
        1-a b & & & y+b
    \end{pmatrix}.
\]

Analogously to the standard case, we extend the operator to the case of a row with multiple columns, and introduce the dual row transfer operators $\widehat{A}, \widehat{B}, \widehat{C}, \widehat{D}$, and the monodromy matrix $\widehat{M}(x,y; a,b)$. We do not fill the details. 

\subsection{The Yang-Baxter equations}
What makes the subject of the integrable lattice models different from other combinatorial analogues like tableaux, Gelfand-Tsetlin patterns, or crystal bases is the existence of the additional structure provided by the Yang-Baxter equation. Informally, the Yang-Baxter equation is a set of relations on the weights of the vertices that guarantees that the row transfer operators satisfy enough functional relations to make it possible to compute the partition functions \emph{exactly}. Because of this reason, integrable lattice models are also called \emph{exactly solvable} lattice models. 

The surprising feature of the Yang-Baxter equation is that the equation itself is a local equation involving weights of only three vertices at a time, but the consequences are the global functional equations for the row transfer matrices with any number of columns. In other words, ``local structure'' implies ``global structure''.

There are two ways to describe the Yang-Baxter equation. One is the graphical way that involves a new type of vertices, called \emph{cross vertices}, and their weights. These new vertices can be attached to the grid of the six vertex model and be moved around to give the functional equations for the partition functions. For this combinatorial description, see Lemma 1 in \cite{BBF11} or Figure 6 in \cite{ABPW21}. We instead describe the Yang-Baxter equation in terms of the functional relations for the row transfer operators. 

We introduce three operators acting on the parametrized vector spaces associated with the rows of the six vertex model:
\begin{align*}
	R^1(x_1,y_1; x_2,y_2)\colon& V(x_1,y_1) \otimes V(x_2,y_2) \to V(x_2,y_1) \otimes V(x_1, y_2),\\
	R^2(x_1,y_1; x_2,y_2)\colon& V(x_1,y_1) \otimes V(x_2,y_2) \to V(x_1,y_2) \otimes V(x_2, y_1),\\
	R(x_1,y_1; x_2,y_2)\colon& V(x_1,y_1) \otimes V(x_2,y_2) \to V(x_2,y_2) \otimes V(x_2, y_2). 	
\end{align*}
In the standard basis, these operators are given by matrices
\begin{align*}
	R^1(x_1,y_1; x_2,y_2) &= \begin{pmatrix}
        x_1 + y_1 & & &\\
         & x_2 + y_1 & 0 &\\
         & x_1-x_2 & x_1+y_1 & \\
         &  &  & x_2 + y_1
    \end{pmatrix},\\
    R^2(x_1,y_1; x_2,y_2) &= \begin{pmatrix}
        x_2 + y_2 & & &\\
         & x_2 + y_2 & y_2-y_1 &\\
         & 0 & x_2+y_1 & \\
         &  &  & x_2 + y_1
    \end{pmatrix},\\
    R(x_1, y_1; x_2, y_2) &= \begin{pmatrix}
        x_1 + y_2 & & &\\
         & x_2 + y_2 & y_2-y_1 &\\
         & x_1-x_2 & x_1+y_1 & \\
         &  &  & x_2 + y_1
    \end{pmatrix}.
\end{align*}
We note that the matrix $R$ factors into the product of two five vertex matrices (see Section 4.7 of \cite{WZ18} for another instance of factorization into the product of two five vertex matrices):
\[
    R(x_1,y_1; x_2,y_2) = (x_2+y_2)^{-1}R^1(x_1,y_2; x_2,y_1)R^2(x_1,y_1; x_2,y_2).
\] 

We will show that these matrices satisfy the Yang-Baxter equations (also called $RRR = RRR$ relations). Traditionally, the solutions of this form of the Yang-Baxter equation are called \emph{$R$-matrices}. 

\begin{theorem}[The refined Yang-Baxter equations]
	{
	\small
	\begin{align*}
		R^1(x_1,y_1; x_2,y_2)R(x_1,y_1; x_3,y_3)R(x_2,y_2; x_3,y_3) &= R(x_2,y_1; x_3,y_3)R(x_1,y_2; x_3, y_3)R^1(x_1,y_1; x_2,y_2),\\
		R^2(x_1,y_1; x_2,y_2)R(x_1,y_1; x_3,y_3)R(x_2,y_2; x_3,y_3) &= R(x_1,y_2; x_3,y_3)R(x_2,y_1; x_3, y_3)R^2(x_1,y_1; x_2,y_2),\\
		R(x_1,y_1; x_2,y_2)R(x_1,y_1; x_3,y_3)R(x_2,y_2; x_3,y_3) &= R(x_2,y_2; x_3,y_3)R(x_1,y_1; x_3, y_3)R(x_1,y_1; x_2,y_2),
	\end{align*}}
	where all equalities are meant as operators from $V(x_1,y_1) \otimes V(x_2,y_2) \otimes V(x_3,y_3)$ to the tensor product of the parametrized vector spaces that is determined by the applied operators. 
\end{theorem}
\begin{proof}
	It is a direct computation and comparison of entries of two $8 \times 8$ matrices. Note that the equations for $R^1$ and $R^2$ imply the equation for $R$ thanks to the factorization.
\end{proof}

We also notice that the $R$-matrices have natural inverses:
\begin{proposition}
	Let $I_4$ be the identity matrix of size $n$. 
	\begin{align*}
		R^1(x_1,y_1; x_2,y_2)R^1(x_2,y_1; x_1,y_2) &= (x_1+y_1)(x_2+y_1)I_4,\\
		R^2(x_1,y_1; x_2,y_2)R^2(x_1,y_2; x_2,y_1) &= (x_2+y_1)(x_2+y_2)I_4,\\
		R(x_1,y_1; x_2,y_2)R(x_2,y_2; x_1,y_1) &= (x_1+y_2)(x_2+y_1)I_4.
	\end{align*}
\end{proposition}
\begin{proof}
	By direct multiplication. 
\end{proof}

We see that these $R$-matrices satisfy a version of the braid relations with two sets of parameters. A remarkable property is that matrices $R^1$ and $R^2$ move only one parameter at a time, and the classical $R$-matrix $R$ factors into $R^1$ and $R^2$. In the classical Yang-Baxter equations, the $R$-matrices only exchange the labels of the parametrized vector spaces which in our case would correspond to the simultaneous exchange of parameters $(x_1,y_1)$ and $(x_2,y_2)$. Because our new matrices allow to change only one parameter at a time, we call them the solutions of the \emph{refined} Yang-Baxter equation. 

\begin{remark}
	By setting $y_i = -q^2 x_i$, the $R$-matrix $R(x_1,y_1; x_2,y_2)$ transforms into the well-known $R$-matrix for the standard evaluation representations of the affine quantum supergroup $U_q(\widehat{\mathfrak{sl}(1|1)})$:
    
    \[
        R_q(x_1,x_2) = \begin{pmatrix}
            x_1 - q^2 x_2 & & &\\
             & x_2 - q^2 x_2 & q^2(x_1-x_2) &\\
             & x_1-x_2 & x_1 - q^2 x_1 & \\
             &  &  & x_2 - q x_1
        \end{pmatrix}.
    \]

    More generally, by setting $y_i = t_i x_i$, we obtain the $R$-matrix from \cite{BBF11} up to a change of variables and rescaling. Furthermore, it is possible to relate these $R$-matrices to the the $R$-matrix given by weights (2.6) from \cite{ABPW21} through similar transformations. 
\end{remark}

\begin{remark}
    The five vertex $R$-matrices $R^1$ and $R^2$ can further be simplified:
    \begin{align*}
        R^1(x_1,y_1; x_2,y_2) &= (x_1+y_1)I_4 - (x_1-x_2)E,\\
        R^2(x_1,y_1; x_2,y_2) &= (x + y_2)I_4 + (y_1-y_2)E',
    \end{align*}
    where 
    \[
    	E = \begin{pmatrix}
            0 & & &\\
            & 1& 0&\\
            & -1 & 0 &\\
            & & & 1
        \end{pmatrix}, \quad E' = \begin{pmatrix}
            1 & & &\\
            & 0& -1&\\
            & 0 & 1 &\\
            & & & 1
        \end{pmatrix}.
    \]
    
    Consider elements $E_i$ and $E'_i$ for $1 \leq i \leq n-1$ acting on $(\C^2)^{\otimes n}$, where the action is on the $i$-th and $(i+1)$-th sites, and by identity elsewhere. Then these generators satisfy the braid relations:
    \[
    	(E_i)^2 = E_i, \quad (E'_i)^2 = E'^2.
    \]
    \[
    	E_i E_{i+1}E_i = E_{i+1} E_{i}E_{i+1}, \quad E'_i E'_{i+1}E_i = E'_{i+1} E'_{i}E'_{i+1}.
    \]
    \[
    	E_{i}E_{j} = E_{j}E_{i}, \quad E'_{i}E_{j}' = E_{j'}E_{i}', \quad \text{for $|i-j| > 1$}. 
    \]
    Moreover, we have
    \[
    	E_i E_i' = E_i', \quad E_i'E_i = E_i.
    \]
    Hence, the algebra of the $R$-matrices $R^1(x_1,y_1; x_2,y_2)$ and $R^2(x_1,y_1; x_2,y_2)$ on $(\C^2)^{\otimes n}$ is an algebra generated by $E_i$'s and $E_i'$'s with the relations above. This algebra can be thought of as a distant free fermionic relative of the Temperley Lieb algebra. \end{remark}

The main application of the $R$-matrices in this paper is that they solve the star-triangle equation (also called $RTT = TTR$ relation) for the operators $T(x,y; a,b)$ which represent a vertex in the six vertex model. Confusingly, it is also often called the Yang-Baxter equation despite having a different form from the equations above. 
 
\begin{theorem}[The refined star-triangle equations]\label{thm:refinedstartriangle}
	Let $a,b$ be single parameters. Write $T(x,y)$ for $T(x,y; a,b)$ for brevity. Then
	\begin{align*}\label{eq:ybequation}
		R^1(x_1,y_1;x_2,y_2)T(x_1,y_1)T(x_2,y_2) &= T(x_2,y_1)T(x_1,y_2)R^1(x_1,y_1; x_2,y_2),\\
		R^2(x_1,y_1;x_2,y_2)T(x_1,y_1)T(x_2,y_2) &= T(x_1,y_2)T(x_2,y_1)R^2(x_1,y_1; x_2,y_2),\\
		R(x_1,y_1;x_2,y_2)T(x_1,y_1)T(x_2,y_2) &= T(x_2,y_2)T(x_1,y_1)R(x_1,y_1; x_2,y_2)
	\end{align*}
	where all equalities are meant as operators from $V(x_1,y_1) \otimes V(x_2,y_2) \otimes W(a,b)$ to the tensor product of the parametrized vector spaces that is determined by the applied operators. 
\end{theorem}
\begin{proof}
	It is a direct computation and comparison of entries of two $8 \times 8$ matrices.
\end{proof}

Informally, the $R$-matrix allows one to exchange the operators $T(x_1,y_1)$ and $T(x_2,y_2)$ when it ``passes'' through them in the Yang-Baxter equation. The miraculous thing is that the $R$-matrix does not depend on parameters $a,b$! Hence, one can use the $R$-matrix to exchange two row operators with arbitrary number of columns.

\begin{corollary}
	The Yang-Baxter equations from \Cref{thm:refinedstartriangle} are true when $a = (a_1,\dots,a_m)$ and $b = (b_1,\dots,b_m)$, in which case $T(x_1,y_1)$ and $T(x_2,y_2)$ are the row operators with $m$ columns labeled by $(a_1,b_1), \dots, (a_m,b_m)$. 
\end{corollary}
\begin{proof}
	By repeating application of \Cref{thm:refinedstartriangle}.
\end{proof}

For our applications, it is particularly convenient to express the Yang-Baxter equation in terms of the functional equations for the row transfer operators. We write $\check{R}(x_1,y_1; x_2,y_2) = R(x_1,y_1; x_2,y_2)P$, where $P$ is the swap operator.

\begin{corollary}\label{lem:ybmonodromy}
	\begin{align*}
		\check{R}^1(x_1,y_1; x_2,y_2)\left(M(x_1,y_1) \otimes M(x_2,y_2)\right) &= \left(M(x_2,y_1) \otimes M(x_1,y_2)\right)\check{R}^1(x_1,y_1; x_2,y_2),\\
		\check{R}^2(x_1,y_1; x_2,y_2)\left(M(x_1,y_1) \otimes M(x_2,y_2)\right) &= \left(M(x_1,y_2) \otimes M(x_2,y_1)\right)\check{R}^2(x_1,y_1; x_2,y_2),\\
		\check{R}(x_1,y_1; x_2,y_2)\left(M(x_1,y_1) \otimes M(x_2,y_2)\right) &= \left(M(x_2,y_2) \otimes M(x_1,y_1)\right)\check{R}(x_1,y_1; x_2,y_2),
	\end{align*}
	where the equality is meant as the entrywise equality of two matrices in the standard basis, and the tensor product is meant to be non-commutative.
\end{corollary}

It is hard not to appreciate the compact form of these equations, where each equation encode $16$ functional relations for the row transfer operators $A,B,C,D$. Let us immediately demonstrate the power of these functional relations by writing down some of these relations.

\begin{proposition}\label{lem:operatorrelations} We have the following properties of the row transfer operators. 
	\item The operators $A,D$ have separate symmetry in $x$'s and $y$'s:
	\begin{align*}
		A(x_1,y_1)A(x_2,y_2) &= A(x_2,y_1)A(x_1, y_2) = A(x_1,y_2)A(x_2,y_1) = A(x_2,y_2)A(x_1,y_1),\\
		D(x_1,y_1)D(x_2,y_2) &= D(x_2,y_1)D(x_1, y_2) = D(x_1,y_2)D(x_2,y_1) = D(x_2,y_2)D(x_1,y_1).
	\end{align*}
	
	\item The operators $B,C$ have partial symmetry in $x$'s and $y$'s:
        \begin{align*}
            (x_1+y_1)B(x_1,y_1)B(x_2,y_2) &= (x_2+y_1)B(x_2,y_1)B(x_1,y_2),\\
            (x_2+y_2)B(x_1,y_1)B(x_2,y_2) &= (x_2+y_1)B(x_1,y_2)B(x_2,y_1),\\
            (x_2+y_1)C(x_1,y_1)C(x_2,y_2) &= (x_1+y_1)C(x_2,y_1)C(x_1,y_2),\\
            (x_2+y_1)C(x_1,y_1)C(x_2,y_2) &= (x_2+y_2)C(x_1,y_2)C(x_2,y_1).
        \end{align*}
\end{proposition}
\begin{proof}
	By equating entries of both sides of \Cref{lem:ybmonodromy}.
\end{proof}

In other words, now we know that some of the row transfer operators \emph{commute} or commute up to a factor. These facts are absolutely not obvious from the combinatorial definition of the six vertex model. Moreover, the refined Yang-Baxter equations allows us to find that there is a \emph{separate symmetry} in parameters $x$'s and $y$'s which is impossible to achieve using the standard Yang-Baxter equation because it exchanges \emph{both} parameters at the same time.

Generally, by reading all entries from \Cref{lem:ybmonodromy}, one gets multiple functional relations for the operators $A,B,C,D$. We do no write them in full because we do not use most of them in this paper. However, we note that these operator relations refine the relations given in Proposition 2.4 of \cite{ABPW21} or Corollary 4.3 of \cite{Kor21} because we get all relations coming from the partial symmetry provided by the refined Yang-Baxter equations. 

Similarly, one can define the $R$-matrices $\widehat{R}^1, \widehat{R}^2, \widehat{R}$ that solve the similar refined Yang-Baxter equations ($RRR = RRR$ type) and the star-triangle equations ($RTT = TTR$) for the dual operators $\widehat{T}(x,y)$ and for the monodromy matrix $\widehat{M}(x,y)$. We omit the details.

There is another class of the interactions between the vertices that is captured by the Yang-Baxter equation that connects the standard vertices with the dual ones. The results of this type occurred in \cite{BBF11}, for the $R$-matrices connecting gamma and delta ice.

Consider the following operator $\overline{R}(x,y; z,w)$ given in the standard basis by the matrix
\begin{equation}
	\overline{R}(x,y; z,w) = \begin{pmatrix}
        1- y w& & & z+w\\
         &  & 1+w x &\\
         & 1+ y z &  & \\
        x+y &  &  & -(1-x z)
    \end{pmatrix}.
\end{equation}

\begin{theorem}[The Yang-Baxter equation of mixed type]
	\[
		\overline{R}(x,y; z,w)T(x,y;a,b)\widehat{T}(z,w;a,b) = \widehat{T}(z,w;a,b)T(x,y;a,b)\overline{R}(x,y;z,w).
	\]
\end{theorem}
\begin{proof}
	Direct computation. 
\end{proof}

As before, we can reformulate the result in terms of the row transfer operators. 

\begin{corollary}\label{lem:dualybmonodromy}
	\[
		\check{\overline{R}}(x,y; z,w)\left(M(x,y) \otimes \widehat{M}(z,w)\right) = \left(\widehat{M}(z,w) \otimes M(x,y)\right)\check{\overline{R}}(x,y; z,w).
	\]
\end{corollary}

Again, this compact equality is actually $16$ relations for the operators $A,B,C,D$ and $\widehat{A},\widehat{B},\widehat{C},\widehat{D}$. Let us write down some of these relations. 

\begin{proposition}
	\small
	\begin{align*}
		(1-yw)A(x,y)\widehat{A}(z,w) + (z+w)C(x,y)\widehat{C}(z,w) &= (1-yw)\widehat{A}(z,w)A(x,y) + (x+y)\widehat{B}(z,w)B(x,y),\\
		(1-x z)\widehat{A}(z,w)A(x,y) + (z+w)\widehat{C}(z,w) C(x,y) &= (1-x z)A(x,y)\widehat{A}(z,w) + (x+y)B(x,y)\widehat{B}(z,w),\\
		(1-x z)D(x,y)\widehat{D}(z,w)+(z+w)\widehat{C}(z,w)C(x,y) &= (1- x z)\widehat{D}(z,w)D(x,y) +  (x+y)B(x,y)\widehat{B}(z,w) ,\\
		(1- yw)\widehat{D}(z,w)D(x,y) + (z+w)C(x,y)\widehat{C}(z,w) &= (1- y w)D(x,y)\widehat{D}(z,w) + (x+y)\widehat{B}(z,w)B(x,y).
	\end{align*}
	
	By combining these relations, we also get
	\[
		A(x,y) \widehat{A}(z,w) - \widehat{A}(z,w)A(x,y) = \widehat{D}(z,w)D(x,y) - D(x,y)\widehat{D}(z,w).
	\]
\end{proposition}

\begin{proposition}\label{lem:finiteoperatorcauchy}
	We have the following relation for the row transfer operators:
	\small
	\begin{multline*}
		(1+y z)(1+x w)\widehat{A}(z,w)A(x,y) - (1-x z)(1-y w)A(x,y)\widehat{A}(z,w) =\\
		= (1-y w)(x+y)B(x,y)\widehat{B}(z,w) +(1-x z)(z+w)C(x,y)\widehat{C}(z,w) + (x+y)(z+w)D(x,y)\widehat{D}(z,w).
	\end{multline*}
\end{proposition}
\begin{proof}
	Follows by algebraic manipulations from the relations in \Cref{lem:dualybmonodromy} when the following free fermionic condition is used in the calculations: 
	\[
		(x+y)(w+z) = (1+wx)(1+yz) - (1-yw)(1-xz).\qedhere
	\]
\end{proof}

The $R$-matrix $\overline{R}(x,y; z,w)$ also satisfies the Yang-Baxter equation of type $RRR = RRR$ and can be factored into the product of two five vertex matrices which gives the refined relations between the standard operators and their duals. However, we omit the details. 

At the end of the day, we have an algebra generated by the row transfer operators $A(x,y),B(x,y),C(x,y),D(x,y)$ and their duals $\widehat{A}(x,y),\widehat{B}(x,y), \widehat{C}(x,y), \widehat{D}(x,y)$. These operators satisfy multiple relations coming from the refined Yang-Baxter equations for the standard vertices, for their dual, and for the mixed type. This algebra naturally acts on the space of columns $W(a,b)$ for any $a = (a_1,\dots,a_n)$ and $b = (b_1,\dots,b_n)$. 	The algebra of this kind is sometimes called the Yang-Baxter algebra. This algebra has a rich structure that can be exploited to ``solve'' the six vertex model. In particular, the algebraic Bethe ansatz is a powerful technique for finding explicit expressions for the matrix coefficients of the row transfer operators in terms of solutions to the Bethe equations. These solutions can then be used to compute the partition function and correlation functions of the six vertex model. These methods were used with great success in \cite{ABPW21}, Appendix A, where formulas for the partition functions were found using the algebraic Bethe Ansatz methods. It would be interesting to explore if the new refined relations could improve this approach.

\subsection{Partition functions}\label{sec:partitionfunctions}
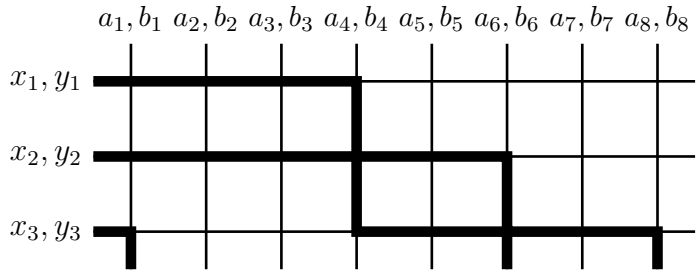
\begin{figure}
    \centering
    \begin{tikzpicture}[xscale=0.5, yscale=0.5]
        \foreach \x in {1,3,5} {
            \draw (0,\x) -- (16,\x);
            \pgfmathtruncatemacro{\index}{4-(\x+1)/2}
            \node [left] at (0,\x) {$x_\index, y_\index$};
        }
        \foreach \y in {1,3,5,7,9,11,13,15} {
            \draw (\y,0) -- (\y,6);
            \pgfmathtruncatemacro{\index}{(\y+1)/2}
            \node [above] at (\y,6) {\small $a_\index, b_\index$};
        }
        
        \draw[path] (0, 1) -- (1, 1) -- (1, 0);
        \draw[path] (0, 3) -- (7, 3) -- (7, 1) -- (11, 1) -- (11, 0);
        \draw[path] (0, 5) -- (7, 5) -- (7,3) -- (11, 3) -- (11, 1) -- (15, 1) -- (15, 0);
    \end{tikzpicture}
    \caption{A typical state in a model for $Z_\alpha$ with $n=3$ and $\alpha = (8, 6, 1)$.}
    \label{fig:modeltype2}
\end{figure} 

In this section we finally put the functional relations for the row transfer operators into practice and find the explicit expressions for the partition functions with various boundary conditions. This section demonstrates the unique feature of the integrable lattice models: while they are combinatorial objects, thanks to integrability, it is possible to find exact expressions for their partition functions.

Let us compute the partition function of the following model. We have $n$ rows labeled by $(x_1,y_1), \dots, (x_n,y_n)$ from top to bottom, and $M$ columns labeled by $(a_1,b_1), \dots, (a_M,b_M)$ from left to right. Then $n$ paths enter the grid on the left at each row, and leave the grid at the bottom columns with the positions $\alpha = (\alpha_1,\dots,\alpha_n)$ read right to left. See \Cref{fig:modeltype2} for the illustration. Let $Z_\alpha$ denote the partition function of this model. We find a determinant expression for $Z_\alpha$ following the method described in Theorem 5 of \cite{BBF11}. Our result generalizes Theorem 5 of \cite{BBF11} as well as Theorem 3.9 of \cite{ABPW21}.

\begin{proposition}\label{lem:determinantfortype2}
	\begin{equation}
		\frac{Z_{\alpha}(x, y; a,b)}{\prod_{i=1}^{n}\prod_{j=1}^{M}(1-b_jx_i)} = \prod_{i=1}^{n}(x_i+y_i) \prod_{i<j}\frac{x_i+y_j}{x_i-x_j}\det\left(\frac{\prod_{k=1}^{\alpha_j-1}(x_i-a_j)}{\prod_{k=1}^{\alpha_j}(1-b_j x_i)}\right)_{1 \leq i,j \leq n}.
	\end{equation}
\end{proposition}
\begin{proof}
	We note that since all weights \eqref{eq:ffweights} in the six vertex model are polynomials in $x,y,a,b$'s, the partition function is also a polynomial in these variables, and hence an element of a unique factorization domain. The partition function $Z_\alpha$ can be equivalently rewritten in terms of the row transfer operators:
	\[
		Z_\alpha = \braket{e_0^{\otimes M} | C(x_1,y_1)C(x_2,y_2)C(x_3,y_3) | e_{m_1(\alpha)} \otimes \dots \otimes e_{m_M(\alpha)}},
	\]
	where $m_i(\alpha) = 1$ if $i \in \alpha$ and $0$ otherwise.

	By inspection, we have at least one vertex of type $\svc_2$ at each row contributing the factor $(x_i+y_i)$ to each state. By the operator relations in \Cref{lem:operatorrelations}, this matrix element is divisible by $(x_i+y_j)$ for all $1 \leq i < j \leq n$. Since all of these polynomials are co-prime, the matrix element $Z_\alpha$ is actually divisible by the product $\prod_{1 \le i < j \leq n}(x_i+y_j)$. Counting the degrees of $y_i$'s in $Z_\alpha(x,y; a,b)$, we conclude that the following ratio is independent on $y$'s:
	\[
		\frac{Z_{\alpha}(x,y;a,b)}{\prod_{i=1}^{n}(x_i+y_i)\prod_{1 \leq i<j \leq n}(x_i+y_j)}
	\]
	Hence, we can set $y_i = -x_i$ for all $1 \leq i \leq n$. Notice that the weight $\sva_2(x,y;a,b) = (y-a)$ becomes $-(x-a)$, and the weight $c_2(x,y;a,b) = x+y$ becomes $0$. Hence, the only states that contribute to the partition function are the ones with no vertices of type $\svc_2$ which are in bijection with the permutation group, and vertices of type $\sva_2$ contribute the sign of the permutation. Hence, the partition function is equal to the symmetrizer of the state with no vertices of type $\svc_2$ and $\sva_2$. There is only one such state when path starting at row $i$ moves right till column $\alpha_i$, and then goes directly down. The weight of this state is 
	\[
		\prod_{i=1}^{n}\prod_{j=1}^{\alpha_i-1}(x_i-a_j)\prod_{j=\alpha_i+1}^{M}(1-b_j x_i).
	\]
	By pulling out all factors $\prod_{j=1}^{M}(1-b_jx_i)$ from the determinant, we get the result. 
\end{proof}

In the special case when $M=n$ and $\alpha = (n,n-1,\dots,1)$, one can go further and compute the determinant explicitly. 

\begin{proposition}[Vandermonde-type determinant]\label{lem:vandermondetype}
	\begin{equation}
		\det\left(\frac{\prod_{k=1}^{n-j}(x_i-a_j)}{\prod_{k=1}^{n-j+1}(1-b_j x_i)}\right)_{1 \leq i,j \leq n} = \frac{\prod_{1 \leq i < j \leq n}(1-a_ib_j)(x_i-x_j)}{\prod_{i,j=1}^{n}(1-b_j x_i)}.
	\end{equation}
\end{proposition}
\begin{proof}
	Multiply both sides by the denominator of the right side and by inspection of the entries in the determinant, observe that both sides are polynomials in $x_1,\dots,x_n$ of the same degrees. After factoring, we only need to figure out the constant which is easily seen to be equal to $\prod_{i < j}(1 - a_ib_j)$ as the coefficient of the leading term $x_1^{n-1}\dots x_{n-1}^1 x_n^0$.
\end{proof}

\begin{remark}
	Similarly, one gets a better-looking determinant identity
	\[
		\det\left(\frac{(x_i - a_1)\dots (x_i-a_{n-j})}{(x_i-b_1)\dots (x_i-b_{n-j+1})}\right)_{1 \leq i,j \leq n} = \frac{\prod_{1 \leq i<j \leq n}(a_i - b_j)(x_i - x_j)}{\prod_{i,j=1}^{n}(x_i - b_j)}.
	\]	
\end{remark}

The special case when $M = n$ and $\alpha = (n,n-1,\dots,1)$ is called the \emph{domain wall boundary conditions}. In this case, we can compute the partition function as a closed product. 

\begin{theorem}[Domain Wall Boundary Conditions]\label{thm:dwbc}
	Let $M = n$ and $\rho_n = (n,n-1,\dots,1)$. Then
	\begin{equation}
		Z_{\rho_n;a,b}(x;y) = \prod_{i=1}^{n}(x_i+y_i) \prod_{1 \leq i < j \leq n}(x_i+y_j)(1-a_ib_j). 
	\end{equation}	
\end{theorem}
\begin{proof}
	By combining the previous two results.	
\end{proof}

Let us next consider the following model. We have $n$ rows labeled by $(x_1,y_1), \dots, (x_n,y_n)$ from top to bottom, and $M_1+M_2$ columns labeled by $(a_{-M_1+1},b_{-M_1+1}), \dots (a_{M_2}, b_{M_2})$ from left to right. Then $n$ paths enter from the top at any $n$ columns with labels $\alpha = (\alpha_1,\dots,\alpha_n)$ in the range $-M_1+1$ to $0$, and leave at the bottom at any $n$ columns with labels $\beta = (\beta_1,\dots,\beta_n)$ in the range from $1$ to $M_2$. Let $Z^{BR}_{\alpha,\beta}$ be the partition function of this model.

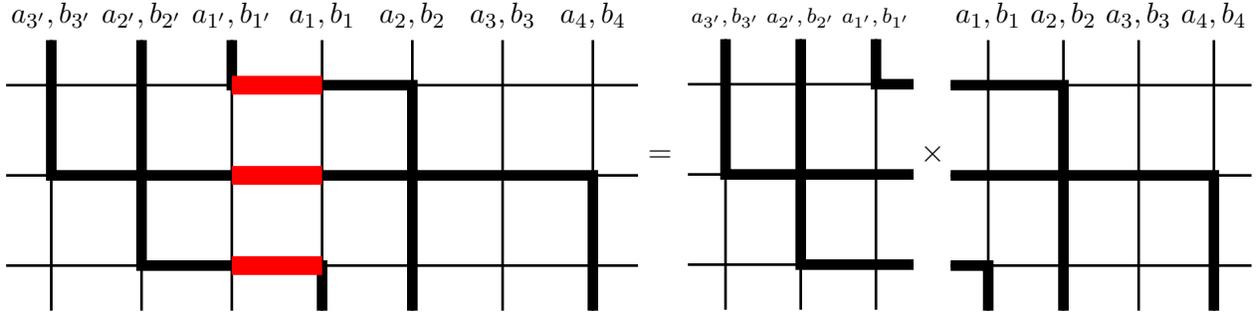
\begin{figure}
    \centering
    \[
    \vcenter{\hbox{\begin{tikzpicture}[xscale=0.6, yscale=0.6]
        \foreach \y in {1,3,5}{
            \draw (0,\y) -- (14, \y);
        }

        \foreach \x in {1,3,5,7,9,11,13}{
            \draw (\x,0) -- (\x, 6);
            
            \ifnum \x > 5
                \pgfmathtruncatemacro{\index}{(\x+1)/2-3}
                \node [above] at (\x, 6) {\small $a_{\index},b_{\index}$};
            \else
                \pgfmathtruncatemacro{\index}{-(\x+1)/2+4}
                \node [above] at (\x, 6) {\small $a_{\index'},b_{\index'}$};
            \fi
        }
        
        \draw[path] (1,6) to (1,3) to (3,3) to (3,1) to (7,1) to (7,0);
        \draw[path] (3,6) to (3,3) to (5,3) to (9,3) to (9,1) to (9,0);
        \draw[path] (5,6) to (5,5) to (9,5) to (9,3) to (13,3) to (13,0);
        \draw[line width=2.5mm, red] (5,5) to (7,5);
        \draw[line width=2.5mm, red] (5,3) to (7,3);
        \draw[line width=2.5mm, red] (5,1) to (7,1);
    \end{tikzpicture}}} = \vcenter{\hbox{\begin{tikzpicture}[xscale=0.5, yscale=0.6]
        \foreach \y in {1,3,5}{
            \draw (0,\y) -- (6, \y);
        }

        \foreach \x in {1,3,5}{
            \draw (\x,0) -- (\x, 6);

            \pgfmathtruncatemacro{\index}{-(\x+1)/2+4}
            \node [above] at (\x, 6) {\tiny $a_{\index'},b_{\index'}$};
        }
        
        \draw[path] (1,6) to (1,3) to (3,3) to (3,1) to (6,1);
        \draw[path] (3,6) to (3,3) to (6,3);
        \draw[path] (5,6) to (5,5) to (6,5);
    \end{tikzpicture}}} \times \vcenter{\hbox{\begin{tikzpicture}[xscale=0.5, yscale=0.6]
        \foreach \y in {1,3,5}{
            \draw (0,\y) -- (8, \y);
        }

        \foreach \x in {1,3,5,7}{
            \draw (\x,0) -- (\x, 6);

            \pgfmathtruncatemacro{\index}{(\x+1)/2}
            \node [above] at (\x, 6) {\small $a_{\index},b_{\index}$};
        }
        
        \draw[path] (0,1) to (1,1) to (1,0);
        \draw[path] (0,3) to (2,3) to (5,3) to (7,3) to (7,0);
        \draw[path] (0,5) to (3,5) to (3,3) to (3,0);
    \end{tikzpicture}}}
    \]
    \caption{The $n$ paths must pass through the horizontal edges in the middle. Then partition function factors into the left and the right sides.}
    \label{fig:factorization}
\end{figure}

\begin{proposition}[Berele-Regev-type factorization]\label{lem:bereleregevf}
	\[
		Z^{BR}_{\alpha,\beta}(x,y; a,b) = Z^*_{\beta}(x,y; a,b)Z_{\alpha}(x,y; a,b),
	\]
	where $Z_\alpha$ is the partition function with $M_2$ columns, and $Z_\beta^*$ is the partition function as shown in \Cref{fig:factorization}. 
\end{proposition}
\begin{proof}
	The proof easily follows from the combinatorial description of the six vertex model. Due to the boundary conditions of the model, all admissible states must have the $n$ paths crossing through the edges in the middle. This means that the partition function of the model factors into the product of two partition functions for the left and the right sides. These partition functions are given by exactly by the factors on the right side.	
\end{proof}

We remark that the proof is unexpectedly simple because of the combinatorial nature of the six vertex model. The similar proof in terms of the row transfer operators or in terms of other combinatorial objects related to the Berele-Regev factorization is much harder. For example, the involution that swaps all occupied edges to unoccupied can be expressed via relations of the weight functions $\sva_1(x,y; a,b) \leftrightarrow \sva_2(x,y; a,b)$, $\svb_1(x,y; a,b) \leftrightarrow \svb_2(x,y; a,b)$, and $\svc_1(x,y; a,b) \leftrightarrow \svc_2(x,y; a,b)$. Then with a little care, one can automatically conclude the dual versions of the results above. In particular, one can express $Z_\beta^*(x,y; a,b)$ in terms of $Z_\alpha(x,y; a,b)$. However, we omit details as it is out of the scope of this paper. 

In the special case, we get the following two-sided generalization of the partition function with the domain wall boundary conditions. 

\begin{proposition}\label{lem:boxvalue}
	Let $\alpha = (n,n-1,\dots,1)$ and $\beta = (n,n-1,\dots,1)$. Then
	\[
		Z^{BR}_{\alpha,\beta}(x,y; a,b) = \prod_{i,j=1}^{n}(x_i+y_j)\prod_{i<j}(1-a_ib_j)(1-a_i'b_j').
	\]
\end{proposition}
\begin{proof}
	We note that both left and right side factorize into the domain wall boundary conditions given by \Cref{thm:dwbc} and the dual version. Then we take the product.
\end{proof}

Let us consider an example of the interaction between the standard six vertex model and the dual version to get us more comfortable working with both types at the same time. When we show both kinds of the models graphically, we draw circles for the vertices of the standard model and squares for the dual ones. 

We consider a grid with just two rows and $m$ columns. The first row is a standard row with the spectral parameter $(x,y)$. The second row is a dual row with the spectral parameter $(z,w)$. The columns have parameters $(a_1,b_1) ,\dots, (a_m,b_m)$. The boundaries on top and bottom are empty. The boundaries on the left are both occupied, and on the right are both unoccupied. The model with a typical admissible state for $m=8$ is shown below:
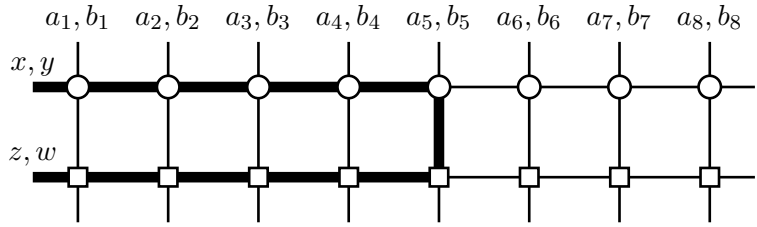
\[
   \begin{tikzpicture}[scale=0.6]
        \draw (1, -2) -- (15, -2);
        \draw (1, 0) -- (15, 0);
        
        \draw (1,-3) -- (1, -2);
        \draw (3,-3) -- (3, -2);
        \draw (5,-3) -- (5, -2);
        \draw (7,-3) -- (7, -2);
        \draw (9,-3) -- (9, -2);
        \draw (11,-3) -- (11, -2);
        \draw (13,-3) -- (13, -2);
        \draw (15,-3) -- (15, -2);
        
        \draw (1,0) -- (1, 1);
        \draw (3,0) -- (3, 1);
        \draw (5,0) -- (5, 1);
        \draw (7,0) -- (7, 1);
        \draw (9,0) -- (9, 1);
        \draw (11,0) -- (11, 1);
        \draw (13,0) -- (13, 1);
        \draw (15,0) -- (15, 1);
        
        \draw (1,-2) -- (1, 0);
        \draw (3,-2) -- (3, 0);
        \draw (5,-2) -- (5, 0);
        \draw (7,-2) -- (7, 0);
        \draw (9,-2) -- (9, 0);
        \draw (11,-2) -- (11, 0);
        \draw (13,-2) -- (13, 0);
        \draw (15,-2) -- (15, 0);
        
        \draw[path] (0,0) -- (9,0) -- (9, -2) -- (0, -2);
        \draw (15,0) to (16,0);
        \draw (15,-2) to (16,-2);
        
        \node[draw, circle, fill=white, inner sep=3pt] at (1, 0) {};
        \node[draw, circle, fill=white, inner sep=3pt] at (3, 0) {};
        \node[draw, circle, fill=white, inner sep=3pt] at (5, 0) {};
        \node[draw, circle, fill=white, inner sep=3pt] at (7, 0) {};
        \node[draw, circle, fill=white, inner sep=3pt] at (9, 0) {};
        \node[draw, circle, fill=white, inner sep=3pt] at (11, 0) {};
        \node[draw, circle, fill=white, inner sep=3pt] at (13, 0) {};
        \node[draw, circle, fill=white, inner sep=3pt] at (15, 0) {};
        
        \node[draw, rectangle, fill=white, inner sep=3.5pt] at (1, -2) {};
        \node[draw, rectangle, fill=white, inner sep=3.5pt] at (3, -2) {};
        \node[draw, rectangle, fill=white, inner sep=3.5pt] at (5, -2) {};
        \node[draw, rectangle, fill=white, inner sep=3.5pt] at (7, -2) {};
        \node[draw, rectangle, fill=white, inner sep=3.5pt] at (9, -2) {};
        \node[draw, rectangle, fill=white, inner sep=3.5pt] at (11, -2) {};
        \node[draw, rectangle, fill=white, inner sep=3.5pt] at (13, -2) {};
        \node[draw, rectangle, fill=white, inner sep=3.5pt] at (15, -2) {};
        
        \node [above] at (0,0) {\small$x,y$};
        \node [above] at (0,-2) {\small$z,w$};
        \node [above] at (1,1) {\small$a_1,b_1$};
        \node [above] at (3,1) {\small$a_2,b_2$};
        \node [above] at (5,1) {\small$a_3,b_3$};
        \node [above] at (7,1) {\small$a_4,b_4$};
        \node [above] at (9,1) {\small$a_5,b_5$};
        \node [above] at (11,1) {\small$a_6,b_6$};
        \node [above] at (13,1) {\small$a_7,b_7$};
        \node [above] at (15,1) {\small$a_8,b_8$};
    \end{tikzpicture} 
\]

The weight of this state written in two columns for convenience is 
\begin{align*}
	\wt(s) = &(x-a_1)(x-a_2)(x-a_3)(x-a_4)(x+y)(1-b_6 x)(1- b_7 x)(1-b_8 x)\\
	&(z-b_1)(z-b_2)(z-b_3)(z-b_4)(1-a_5b_5)(1-a_6 z)(1-a_7 z)(1-a_8 z).
\end{align*}

It is easy to see that the partition function $Z$ of the model is equal to 
\[
	G_m = \sum_{k=1}^{m}(1-a_kb_k)(x+y)\prod_{i=1}^{k-1}(x-a_i)(z-b_i)\prod_{i=k+1}^{m}(1-b_k x)(1-a_k z).
\]

\begin{proposition}[Sum of the double factorial geometric progression]
    \begin{equation}\label{eq:geomprog}
        \sum_{k=0}^{m}(1-a_{k+1}b_{k+1})\frac{(x|a)^k}{(x;b)^{k+1}}\frac{(z|b)^k}{(z;a)^{k+1}} = \frac{1 - \frac{(x|a)^{m+1}}{(x;b)^{m+1}}\frac{(z|b)^{m+1}}{(z;b)^{m+1}}}{1-xz}.
    \end{equation}
\end{proposition}
\begin{proof}
	We compute $G_m$ in a different way. In terms of the row transfer operators, we have
	\[
		G_m = \braket{e_0^{\otimes m} | C(x,y)\widehat{C}(z,w) | e_0^{\otimes m}}.
	\]
	Now, using the relations for the row transfer operators from \Cref{lem:dualybmonodromy}, we have
	{\small 
	\[
		(x+y)A(x,y)\widehat{A}(z,w) - (1-xz)C(x,y)\widehat{C}(z,w) = (1-yw)\widehat{C}(z,w)C(x,u) + (x+y)\widehat{D}(z,w)D(x,y).
	\]}Apply $\braket{e_0^{\otimes m} | \cdot | e_0^{\otimes m}}$ to both sides. It is easy to see that the matrix element of $\widehat{C}(z,w)C(x,y)$ is zero. By evaluating the rest of the matrix elements explicitly and rearranging the terms, we get $(1-x z)G_m = (x+y)\left((x;b)^m (z;a)^m - (x|a)^m (z|b)^m\right)$. Now divide both sides by $(x+y)(1-xz)(x;b)^m (z;a)^m$ and change $m$ with $m+1$.  
\end{proof}

Using the row transfer operators, involutions of the six vertex model, and the LGV lemma (\Cref{lem:svlgv}), one can compute many more partition functions and find new identities. 

\section{The infinite six vertex model}\label{sec:infinitesv}
Now we extend the six vertex model to the case of infinitely many columns.

Let $a = (a_i)_{i \in \Z}$ and $b = (b_i)_{i \in \Z}$ be two doubly infinite sequences of indeterminates or complex numbers. Let $W(a,b) = \bigoplus_{k \in \Z}W(a_k,b_k)$ be the tensor product of infinitely many column spaces. Let $\mathcal{W}$ be a subspace of $W(a,b)$ spanned by basis vectors $\otimes_{k \in \Z}^{\infty}e_{i_k}$ such that $i_k = 0$ for all large enough $k \gg 0$, and $i_k = 1$ for all large enough $k \ll 0$. 

Let $a^{n,m} = (a_{-n+1}, a_{-n+2}, \dots, a_{m-1},a_m)$ denote a finite subsequence of parameters. Then for each vector $v$, and for large enough integers $n,m$, there exists a vector $v^{n,m} \in W(a^{n,m}, b^{n,m})$ such that
\[
	v = \dots e_{1} \otimes e_{1} \otimes e_1 \otimes v^{n,m} \otimes e_0 \otimes e_0 \otimes e_0 \dots
\]
We usually identify $v$ with the projection $v^{n,m}$ to simplify notation. 

Our goal is to extend the six vertex model to the case of infinitely many columns when considering only the vectors from $\mathcal{W}$ for the space of columns. The idea is to take the limit of the operators with finite number of columns labeled by $(a_{-n+1},b_{-n+1}), \dots, (a_m,b_m)$ as $n,m \to \infty$. In order to avoid infinite products, we normalize the weights of the vertices for each operator such that the vertices that occur infinitely often have weight $1$. In terms of the equations, we define operators $\mathcal{A}_{a,b},\mathcal{B}_{a,b},\mathcal{C}_{a,b},\mathcal{D}_{a,b}$ and their dual versions $\widehat{\mathcal{A}}_{a,b}, \widehat{\mathcal{B}}_{a,b}, \widehat{\mathcal{C}}_{a,b}, \widehat{\mathcal{D}}_{a,b}$ using the matrix coefficients, where we identify $v,w$ with its projections to $W(a^{n,m}, b^{n,m})$: 
{ \footnotesize 
\begin{align*}
    \braket{v | \mathcal{A}(x,y) | w} = \lim_{n,m \to +\infty}\frac{\braket{v | A(x,y;a^{n,m},b^{n,m}) | w }}{(y;b')^{n}(x;b)^{m}}, \quad \braket{v | \widehat{\mathcal{A}}(x,y) | w} &= \lim_{n,m \to +\infty}\frac{\braket{v | \widehat{A}(x,y;a^{n,m},b^{n,m}) | w }}{(y;a')^{n}(x;a)^{m}},\\
    \braket{v | \mathcal{B}(x,y) | w} = \lim_{n,m \to +\infty}\frac{\braket{v | B(x,y;a^{n,m},b^{n,m}) | w }}{(y;b')^n(x|a)^{m}}, \quad \braket{v | \widehat{\mathcal{B}}(x,y) | w} &= \lim_{n,m \to +\infty}\frac{\braket{v | \widehat{B}(x,y;a^{n,m},b^{n,m}) | w }}{(y;a')^{n}(x;b)^{m}},\\
    \braket{v | \mathcal{C}(x,y) | w} = \lim_{n,m \to +\infty}\frac{\braket{v | C(x,y;a^{n,m},b^{n,m}) | w }}{(y|a')^n(x;b)^{m}}, \quad \braket{v | \widehat{\mathcal{C}}(x,y) | w} &= \lim_{n,m \to +\infty}\frac{\braket{v | \widehat{C}(x,y;a^{n,m},b^{n,m}) | w }}{(y|b')^{n}(x;a)^{m}},\\
    \braket{v | \mathcal{D}(x,y) | w} = \lim_{n,m \to +\infty}\frac{\braket{v | D(x,y;a^{n,m},b^{n,m}) | w }}{(x|a)^{m}(y|a')^m}, \quad \braket{v | \widehat{\mathcal{D}}(x,y) | w} &= \lim_{n,m \to +\infty}\frac{\braket{v | \widehat{D}(x,y;a^{n,m},b^{n,m}) | w }}{(y|b')^{n}(x|b)^{m}}.
\end{align*}
}
Note that we omit the parameters $a,b$ from notation. 

It is easy to see that the operators are well-defined on $\mathcal{W}$ because for each operator the weights of vertices that appear infinitely often will cancel out with denominator. Moreover, in each case, for fixed $v,w \in \mathcal{W}$, the limit actually \emph{stabilizes}, so for large enough $n,m$, the matrix coefficient is equal to the normalized matrix coefficient of the finite operator. In other words, the limit should is formal, and all results remain algebraic (and not analytic). 

Combinatorically, the infinite six vertex model corresponds to the extension of the finite six vertex model in both directions so that the columns are parametrized by all integers. Then one extends the row transfer operators by normalizing weights of the vertices. 

From now on, we concentrate only on operators $\mathcal{A}$ and $\widehat{\mathcal{A}}$.

\subsection{Functional relations for the row transfer operators}\label{sec:functionalrelations}
Since the row transfer operators $\mathcal{A}$ and $\widehat{\mathcal{A}}$ are obtained as a limit, many of the properties follow from the finite versions.

\begin{proposition}\label{lem:operatorrelationsforinfa}
    The operators $\mathcal{A}$ and $\widehat{\mathcal{A}}$ satisfy the following properties:
    \begin{enumerate}
        \item The cancellation property: $\mathcal{A}(t,-t) = 1$ and $\widehat{\mathcal{A}}(t,-t) = 1$.
        \item The operators have separate symmetry in $x$'s and $y$'s:
        	\begin{align*}
        		\mathcal{A}(x_1,y_1)\mathcal{A}(x_2,y_2) &= \mathcal{A}(x_2,y_1)\mathcal{A}(x_1, y_2) = \mathcal{A}(x_1,y_2)\mathcal{A}(x_2,y_1) = \mathcal{A}(x_2,y_2)\mathcal{A}(x_1,y_1),\\
        		\widehat{\mathcal{A}}(x_1,y_1)\widehat{\mathcal{A}}(x_2,y_2) &= \widehat{\mathcal{A}}(x_2,y_1)\widehat{\mathcal{A}}(x_1, y_2) = \widehat{\mathcal{A}}(x_1,y_2)\widehat{\mathcal{A}}(x_2,y_1) = \widehat{\mathcal{A}}(x_2,y_2)\widehat{\mathcal{A}}(x_1,y_1).
        	\end{align*}
    \end{enumerate}
\end{proposition}
\begin{proof}
	For any $v,w \in \mathcal{W}$, there exist large enough integers $n,m > 0$ such that the matrix elements of $\mathcal{A}$ and $\widehat{\mathcal{A}}$ are up to normalization are equal to the matrix elements of $A(x,y)$ and $\widehat{A}(x,y)$ with finitely many columns. Then the result follows from the finite case, normalization, and the dual versions by \Cref{lem:vanishing}, \Cref{lem:operatorrelations}.
\end{proof}

The surprising thing is that under certain assumptions, the operator relations can \emph{simplify} in the limit. See the detailed discussion of the similar simplified relations in Section 7 of \cite{ABPW21}. We do not write the full list of simplified relations here.  

\begin{theorem}\label{thm:operatorcauchy}
	The operators $\mathcal{A}$ and $\widehat{\mathcal{A}}$ satisfy the following relation: 
	\begin{equation}
		\mathcal{A}(x,y)\widehat{\mathcal{A}}(z,w) = \frac{1+y z}{1-x z}\frac{1+x w}{1-y w}\widehat{\mathcal{A}}(z,w)\mathcal{A}(x,y),
	\end{equation}
	where the equality is understood in the following sense. For given $v,w \in \mathcal{W}$, there exists large enough integers $k,l$ such that the matrix elements of both sides of the equation are equal given that
	\[
		\frac{y-a_k'}{1-b_k'y}\frac{1-b_k'w}{w-a_k'} = 0 \quad \text{and} \quad \frac{x-a_l}{1-b_l x}\frac{1-b_lz}{z-a_l} = 0.
	\]
\end{theorem}
\begin{proof}
	By \Cref{lem:finiteoperatorcauchy}, 
	\begin{multline*}
		(1+y z)(1+x w)\widehat{A}(z,w)A(x,y) - (1-x z)(1-y w)A(x,y)\widehat{A}(z,w) =\\
		= (1-y w)(x+y)B(x,y)\widehat{B}(z,w) +(1-x z)(z+w)C(x,y)\widehat{C}(z,w) + (x+y)(z+w)D(x,y)\widehat{D}(z,w).
	\end{multline*}
	Let $v,w \in \mathcal{W}$. Then we project $v,w$ to $W(a^{n,m}, b^{n,m})$ for large enough $n,m$ and keep the same notation. We divide both sides by $(y;b)^{n} (x;b)^{m}(w|a')^{n}(z|a)^{m}$ so that the row transfer operators $A$ and $\widehat{A}$ will be converge to $\mathcal{A}$ and $\widehat{\mathcal{A}}$ as $n,m$ goes to infinity. 
	
	By inspection, for large enough $k,l$ that depend on $v,w$, using the graphical presentation of the row transfer operators, we find that all admissible states of the operators on the right side of the identity have either factors 
	\[
		\frac{y-a_k'}{1-b_k'y}\frac{1-b_k'w}{w-a_k'} \quad \text{or} \quad \frac{x-a_k}{1-b_k x}\frac{1-b_kz}{z-a_k},
	\] 
	or both. At the same time, the matrix elements of the operators $A$ and $\widehat{A}$ don't have these factors thanks to normalization. Hence, under the assumption, as $n,m$ goes to infinity, the right side is equal to zero, and the left side converges to the operators $\mathcal{A}$ and $\widehat{\mathcal{A}}$.	
\end{proof}

More generally, we have well-defined operators 
\[
	\mathcal{A}(x_1,y_1)\dots \mathcal{A}(x_n,y_n)\widehat{\mathcal{A}}(z_1,w_1)\dots \widehat{\mathcal{A}}(z_m,w_m),
\]
where the equality is understood in the following sense. For any $v,w \in \mathcal{W}$, there exists large enough integers $k,l$ such that the matrix elements of both sides of the equation are equal given that 
\[
	\frac{y_i-a_k'}{1-b_k'y_i}\frac{1-b_k'w_j}{w_j-a_k'} = 0 \quad \text{and} \quad \frac{x_i-a_l}{1-b_l x_i}\frac{1-b_lz_j}{z_j-a_l} = 0.
\]
In other words, despite a priori there are infinitely many states, we understand the identity in a purely algebraic sense. These conditions can alternatively be formalized by treating the matrix elements on both sides are the formal power series in the double factorial powers, but we do not develop this line.

\begin{remark}
	The result can also be made analytical. See Section 7 of \cite{ABPW21}, in particular the discussion around equations (7.18), (7.19). We set parameters $a = (a_i)_{i \in \Z}$ and $b = (b_i)_{i \in \Z}$, and variables $x,y$ to be complex numbers. Then for convergence we assume the following conditions. Then for convergence, it is enough to assume that 
		\[
			\left|\frac{(x|a)^{m+1}}{(x;b)^{m+1}}\frac{(z|b)^{m+1}}{(z;a)^{m+1}}\right| \to 0.
		\]
		It is sufficient to assume that for infinitely many $j=1,2,\dots$, we have
	\[
		\left|\frac{x-a_j}{1-b_j x}\frac{z-b_j}{1-a_j z}\right| < 1 -\epsilon < 1.
	\]
	
	If $x=(x_1,\dots,x_n)$ and $z = (z_1,\dots,z_m)$, we assume these inequalities for all $x_i$'s and $z_j$'s.
\end{remark}

\subsection{Matrix coefficients in terms of ribbons}
In this section, we describe the matrix coefficients of the operators $\mathcal{A}$ and $\widehat{\mathcal{A}}$ explicitly in the basis of charged partitions. More generally, the entire Yang-Baxter algebra generated by the row transfer operators could be described in terms of ribbons. We omit this description because we do not use them in this paper. See Section 4.3 of \cite{Kor21} for the relevant discussion. 

Recall that $\mathcal{W}$ is spanned by basis vectors $\otimes_{k \in \Z}^{\infty}e_{i_k}$ such that there exist $m_-,m_+$ such that for any $i_k = 0$ for $k \geq m_-$ and $i_k = 1$ for $k \leq m_+$. In other words, the sequence of indices stabilizes by $1$ on the left, and by $0$ on the right. Such binary sequences are called \textit{Maya diagrams}. Let $\lambda$ be a partition and $c \in \Z$ be an integer that we call \emph{level} or \emph{charge}. Then we denote
\begin{equation}
	\alpha_c(\lambda) = \{\lambda_i-i+1+c\}_{i=1}^{\infty}, \quad \beta_c(\lambda) = \{u-\lambda_i' + c\}_{i=1}^{\infty}. 
\end{equation}

It is well-known that (see I, (1.7) in \cite{Mac95}) for any partition $\lambda$, we have $\alpha_c(\lambda) \cap \beta_c(\lambda) = \emptyset$, and
\begin{equation}
	\alpha_c(\lambda) \sqcup \beta_c(\lambda) = \{\lambda_i-i+1+c\}_{i=1}^{\infty} \sqcup \{i-\lambda_i'+c\}_{i=1}^{\infty} = \Z.
\end{equation}

It is not hard to see that the tuples $(\lambda;c)$ for all partitions $\lambda$ and charges $c \in \Z$ produce all Maya diagrams, where the corresponding Maya diagram $(i_k)_{k \in \Z}$ has $i_k = 0$ if $k \in \beta_c(\lambda)$ and $i_k = 1$ if $k \in \alpha_c(\lambda)$. We denote by $\ket{\lambda; c} = \otimes_{k \in \Z}e_{i_k}$ the corresponding basis elements, and by $\bra{\lambda; c}$ the corresponding dual basis. Sometimes these partitions are called \emph{charged partitions}. For the level zero, we write simply $\ket{\lambda}$ and $\bra{\lambda}$. 

Hence, the space $\mathcal{W}$ decomposes into the direct sum $\mathcal{W} = \bigoplus_{c \in \Z}\mathcal{W}_c$, where each $W_c$ is spanned by basis elements $|\lambda; c \rangle$ for all partitions $\lambda$. Each $\mathcal{W}_c$ is called a \emph{level space}.

Let us find the explicit expression for the matrix coefficients of $\mathcal{A}$ and $\widehat{\mathcal{A}}$ in the basis of the charged partitions. The following graphical representations will motivate the results. 

Let $v = \bra{\mu; 3}$ and $w = \ket{\lambda; 3}$ for $\lambda = (6,6,4,2,2)$ and $\mu = (7,6,5,5,3,2,1)$. Then $\braket{v | \mathcal{A}(x,y) | w}$ is graphically shown below. The picture extends on the left with only vertices of type $\svb_1$, and on the right with only vertices of type $\sva_1$. The white circles are normalized so that $\sva_1 = 1$ and the gray circles are normalized so that $\svb_1 = 1$. 
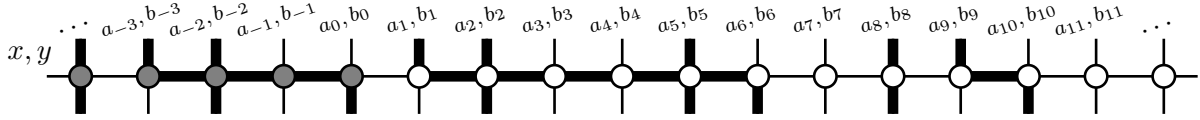
\[
   \begin{tikzpicture}[xscale=0.45, yscale=0.5]
        \draw (1,0) -- (33,0);
        \draw (0,0) -- (1,0);
        \draw (33,0) -- (34,0);
                
       	\draw[path] (1,1) -- (1,-1);
        
        \draw[path] (3,1) -- (3,0) -- (9,0) -- (9,-1);
        \draw[path] (5,1) -- (5,-1);
        
        \draw[path] (11,1) -- (11,0) -- (21,0) -- (21,-1);
        \draw[path] (13,1) -- (13,-1);
        \draw[path] (19,1) -- (19,-1);
        
        \draw[path] (25,1) -- (25,-1);
        
        \draw[path] (27,1) -- (27,0) -- (29,0) -- (29,-1);
        
        \node [above] at (-0.5,0) {\small$x,y$};
        
 		\foreach \i in {1,...,17} {
            \pgfmathtruncatemacro{\x}{2*\i-1}
            \draw (\x,-1) -- (\x,1);
            \pgfmathtruncatemacro{\si}{\i-5}
            
            \ifnum\i<6
		        \ifnum\x=1
		        	\node [above, rotate=20] at (\x,1) {\small $\dots$};
		            \node[draw, circle, fill=gray, inner sep=3pt] at (\x,0) {};
		        \else
		            \node [above, rotate=20] at (\x,1) {\tiny $a_{\si},b_{\si}$};
		            \node[draw, circle, fill=gray, inner sep=3pt] at (\x,0) {};
		        \fi
		    \else
		        \ifnum\x=33
		        	\node [above, rotate=20] at (\x,1) {\small $\dots$};
		            \node[draw, circle, fill=white, inner sep=3pt] at (\x,0) {};
		        \else
		            \node [above, rotate=20] at (\x,1) {\tiny $a_{\si},b_{\si}$};
		            \node[draw, circle, fill=white, inner sep=3pt] at (\x,0) {};
		        \fi
		    \fi
        }
    \end{tikzpicture}
\]
Then the weight is equal to 
{\small 
\[
	\left(\frac{1-a_{-3}b_{-3}}{1+b_{-3}y}\right)\left(\frac{y+a_{-2}}{1+b_{-2}y}\right)\left(\frac{x-a_{-1}}{1+b_{-1}y}\right)\left(\frac{x+y}{1+b_{0}y}\right)\left(\frac{1-a_{1}b_{1}}{1-b_{1}x}\right)\left(\frac{y+a_{2}}{1-b_{2}x}\right)\dots \left(\frac{x+y}{1-b_{10}x}\right)
\]}

From the graphical presentation, we see that the matrix coefficient can be expressed in terms of the product of ``ribbons'', the configurations of paths of the following form:
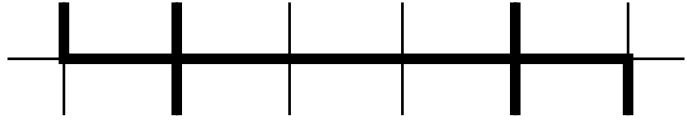
\[
    \begin{tikzpicture}[scale=0.75]
        \draw (0,0) -- (12,0);
        \draw (1,-1) -- (1,1);
        \draw (3,-1) -- (3,1);
        \draw (5,-1) -- (5,1);
        \draw (7,-1) -- (7,1);
        \draw (9,-1) -- (9,1);
        \draw (11,-1) -- (11,1);
        \draw[path] (1,1) -- (1,0) -- (11,0) -- (11,-1);
        \draw[path] (3,1) -- (3,-1);
        \draw[path] (9,1) -- (9,-1);
    \end{tikzpicture}
\]
The name ``ribbon'' is justified by the following
\begin{lemma}\label{lem:ribbonmaya}
	Let $\lambda/\mu$ be a ribbon with the minimal content $k$ and maximal content $K$. Let $i_1,\dots,i_{r-1}$ be the contents of the boxes that have bottom neighbors. Then
	\begin{enumerate}
		\item $k \in \alpha_0(\mu) \cap \beta_0(\lambda) = \alpha_0(\mu) \setminus \alpha_0(\lambda)$,
		\item $i_1,\dots,i_{r-1} \in \alpha_0(\lambda) \cap \alpha_0(\mu)$,
		\item $K+1 \in \alpha_0(\lambda) \cap \beta_0(\mu) = \alpha_0(\lambda) \setminus \alpha_0(\mu)$.
	\end{enumerate}
\end{lemma}
\begin{proof}
	By inspection. 
\end{proof}

\begin{example}
	Let $\lambda = (4,2,2)$ and $\mu = (1,1)$. Here is the skew shape $\lambda/\mu$ together with its contents written in the corresponding boxes:
	\[
	  \ytableausetup{notabloids}
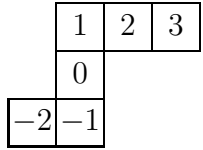
	  \begin{ytableau}
	   \none & 1 & 2 & 3\\
	   \none & 0  \\
		-2 & -1
	  \end{ytableau}
	\]
	Then $\alpha(\lambda) = \{4,1,0,-3,-4,-5,\dots\}$ and $\alpha(\mu) = \{1,0,-2,-3,-5,\dots\}$. Then indeed we have $-2 \in \alpha(\mu) \setminus \alpha(\lambda)$, and $0,1 \in \alpha(\lambda) \cap \alpha(\mu)$, and $4 \in \alpha(\lambda) \setminus \alpha(\mu)$.
\end{example}

\begin{lemma}
	We have $\braket{\mu; c | \mathcal{A}_{a,b}(x/y) | \lambda; c} = 0$ unless $\lambda/\mu$ is a union of ribbons. 
\end{lemma}
\begin{proof}
	The only admissible states correspond to the union of ``ribbons'' in Maya diagrams which correspond to the ribbons in the charged partitions.
\end{proof}

The values on the diagonal can be given explicitly as a product:
\begin{lemma}[Diagonal values]\label{lem:operatoremptydiagram}
	\[
		\braket{\lambda;c | \mathcal{A}_{a,b}(x/y) | \lambda; c} = \prod_{i=1}^{\infty}\frac{1-b_{i-\mu_i'+c}x}{1-b_{i}x}\frac{1+b_{\mu_i-i+1+c}y}{1+b_{-i+1}y}
	\]
	Note that the product is always finite. 
\end{lemma}

Now we can give the formula for the matrix coefficients of the operator $\mathcal{A}$. 

\begin{proposition}\label{lem:operatorexplicitrow}
	We have $\braket{\mu; c | \mathcal{A}_{a,b}(x/y) | \lambda; c} = 0$ if the skew diagram $\lambda/\mu$ has a $2 \times 2$ block of boxes. If $\lambda/\mu$ has no $2 \times 2$ blocks of squares, that is, if it is a union of ribbons, then
	\begin{align}
		\frac{\braket{\mu; c | \mathcal{A}_{a,b}(x/y) | \lambda; c}}{\braket{\mu; c | \mathcal{A}_{a,b}(x/y) | \mu; c}} &= \prod_{r \in R(\lambda/\mu)}\wt_{\tau^c a, \tau^c b}(r;x,y).
	\end{align}

	where the weight $\wt_{a,b}(r;x,y)$ of a ribbon $r \in R(\lambda/\mu)$ is given by:
	\begin{equation}\label{eq:ribbonweight}
		\wt_{a,b}(r; x,y) = \frac{(1-a_kb_k)(x+y)}{(1+b_{k} y)(1-b_{K+1} x)}\prod_{\alpha \in r}\begin{cases}
				\frac{x-a_{c(\alpha)}}{1-b_{c(\alpha)}x}, &\text{if $\alpha$ has the left neighbor},\\
				\frac{y+a_{c(\alpha)}}{1+b_{c(\alpha)}y}, &\text{if $\alpha$ has the bottom neighbor};
			\end{cases}
	\end{equation}
	where $k = \min_{\alpha \in r}(c(\alpha))$ is the minimal content in $r$, and $K = \max_{\alpha \in r}(c(\alpha))$ is the maximal content in $r$. 
\end{proposition}
\begin{proof}
	Let us consider the ratio on the left side of the equation. Thanks to normalization, when we consider the quotient of corresponding weights, we have the uniform description that does not depend whether the vertices have positive or negative indices. Moreover, explicitly, we have vertices of types $\sva_1$ and $\svb_1$ have weight $1$ because they cancel out in the numerator and denominator. The vertices of type $\sva_2$ have weight of the form $(y+a)/(1+by)$, vertices of type $\svb_2$ have weight of the form $(x-a)/(1-bx)$, vertices of type $\svc_1$ have weight of the form $(1-ab)/(1+by)$, and vertices of type $\svc_2$ have weight $(x+y)/(1-bx)$. Since only vertices of types $\svc_1,\sva_2,\svb_2,\svc_2$ contribute non-trivially, the weight factors into the weight over ``ribbons''. 
	
	Let $r$ be a ribbon in $\lambda/\mu$ with minimal content $k$ and maximal content $K$, which corresponds to the vertices of types $\svc_1$ and $\svc_2$ in the Maya representation of the ribbon. Let $i_1,\dots,i_{k}$ be the contents of the boxes with bottom neighbors. Then the weight is given by
	\[
		\left(\frac{1-a_kb_k}{1+b_k y}\right)\left(\frac{x+y}{1-b_{K+1} x}\right)\prod_{\alpha \in r}\begin{cases}
			\frac{x-a_{c(\alpha)}}{1-b_{c(\alpha)}x}, &\text{if $\alpha$ has the left neighbor}\\
			\frac{y+a_{c(\alpha)}}{1+b_{c(\alpha)}y}, &\text{if $\alpha$ has the bottom neighbor}
		\end{cases}
	\]
	This completes the proof.
\end{proof}

Analogously, we prove 
\begin{proposition}
	We have $\braket{\lambda; c | \widehat{\mathcal{A}}_{a,b}(x/y) | \mu; c} = 0$ if the skew diagram $\lambda/\mu$ has a $2 \times 2$ block of boxes. If $\lambda/\mu$ has no $2 \times 2$ blocks of squares, that is, if it is a union of ribbons, then
	\begin{align}
		\frac{\braket{\lambda; c | \widehat{\mathcal{A}}_{a,b}(x/y) | \mu; c}}{\braket{\lambda; c | \widehat{\mathcal{A}}_{a,b}(x/y) | \mu; c}} &= \prod_{r \in R(\lambda/\mu)}\widehat{\wt}_{\tau^c a, \tau^c b}(r;x,y).
	\end{align}

	where the weight $\widehat{\wt}_{a,b}(r;x,y)$ of a ribbon $r \in R(\lambda/\mu)$ is given by:
	\begin{equation}\label{eq:dualribbonweight}
		\widehat{\wt}_{a,b}(r; x,y) = \frac{(1-a_{K+1}b_{K+1})(x+y)}{(1+a_{K+1} y)(1-a_{k}x)}\prod_{\alpha \in r}\begin{cases}
				\frac{x-b_{c(\alpha)+1}}{1-a_{c(\alpha)+1}x}, &\text{if $\alpha$ has the right neighbor},\\
				\frac{y+b_{c(\alpha)+1}}{1+a_{c(\alpha)+1}y}, &\text{if $\alpha$ has the top neighbor},
			\end{cases}
	\end{equation}
	where $k = \min_{\alpha \in r}(c(\alpha))$ is the minimal content in $r$, and $K = \max_{\alpha \in r}(c(\alpha))$ is the maximal content in $r$. 
\end{proposition}

We now establish the following duality between the operators $\mathcal{A}$ and $\widehat{\mathcal{A}}$.
\begin{proposition}[Duality]\label{lem:weightduality}
	We have
	\begin{equation}
		\braket{\mu | \mathcal{A}_{a,b}(x/y) | \lambda} = \braket{\lambda' | \widehat{\mathcal{A}}_{a,b}(x/y) | \mu'}.
	\end{equation}
\end{proposition}
\begin{proof}
	In the conjugated skew diagram $\lambda'/\mu'$, having a left neighbor changes to having a top neighbor, and having a bottom neighbor changes to having a right neighbor. This accounts for the change of $x$ to $y$ and $y$ to $x$, as well as the exchange of $a$ and $b$ for $b'$ and $a'$. The content changes the sign, so the indices in the formula change from $c(\alpha)$ to $-c(\alpha)+1$, which accounts for the exchange of $b$ and $a$ for $b'$ and $a'$. Therefore, the identity follows.
\end{proof}

\section{Schur functions}
In this section, we use the free fermionic six vertex model to define a new family of supersymmetric Schur functions which depends on two sequences of parameters. 

\subsection{Definition, examples, and degenerations}
Let $x = (x_1,\dots,x_n)$ and $y = (y_1,\dots,y_n)$ be two sets of independent indeterminates that we think of as variables. Let $a = (a_i)_{i \in \mathbb{Z}}$ and $b = (b_i)_{i \in \mathbb{Z}}$ be two sequences of independent indeterminates that we think of as parameters. Let $\lambda$ and $\mu$ be two partitions. 

We define the \emph{free fermionic Schur functions} $s_{\lambda/\mu; a,b}(x/y)$ by
\begin{equation}
	s_{\lambda/\mu; a,b}(x/y) =  \braket{\mu|\mathcal{A}_{a,b}(x_1,y_1) \dots \mathcal{A}_{a,b}(x_n,y_n)|\lambda},
\end{equation}
where $\mathcal{A}_{a,b}(x,y)$ are the operators defined in \Cref{sec:infinitesv}. 

As usual, we write $s_{\lambda; a,b}$ for $s_{\lambda/\emptyset; a,b}$. We note that $s_{\lambda/\mu;a,b} = 0$ unless $\mu \subseteq \lambda$. Therefore, we always assume that $\mu \subseteq \lambda$ when we write the Schur functions $s_{\lambda/\mu}$ in this section. 

\begin{example}
	By \Cref{lem:operatoremptydiagram}, we have
	\[
		s_{\lambda/\lambda; a,b}(x/y) = \prod_{i=1}^{n}\prod_{k=1}^{\infty}\frac{1-b_{k - \lambda_k'}x_i}{1-b_k x_i}\frac{1+b_{\lambda_k-k+1}y_i}{1+b_{-k+1}y_i}.
	\]
	In particular, $s_{\emptyset/\emptyset; a,b}(x/y) = 1$. We note that unlike in the classical case, the free fermionic Schur functions for the ``empty'' Young diagram $\lambda/\lambda$ depend on the partition $\lambda$. 
\end{example}

\begin{example}
	Let $x,y$ be single variables, then by \Cref{lem:operatorexplicitrow}, we have the expression in terms of ribbons. If $\lambda/\mu$ contains a $2 \times 2$ block of squares, then $s_{\lambda/\mu;a,b} = 0$. Otherwise,
	\begin{equation}\label{eq:onerow}
		s_{\lambda/\mu;a,b}(x/y) = s_{\mu/\mu;a,b}(x/y)\prod_{r \in R(\lambda/\mu)}\wt_{a,b}(r; x,y),
	\end{equation}
	where the weight of a ribbon $\wt_{a,b}(r;x,y)$ is given by \Cref{sec:infinitesv}.
\end{example}

\begin{example}
	Let $\lambda/\mu = (4,2,2)/(1,1)$. Then $\lambda/\mu$ with its contents is depicted by
	\[
	\ytableausetup{notabloids}
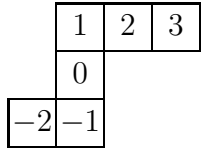
		\begin{ytableau}
			\none & 1 & 2 & 3\\
			\none & 0  \\
			-2 & -1
		\end{ytableau}
	\]
	Then in single variables $x,y$, we have $s_{\lambda/\mu; a,b}(x/y) = s_{\mu/\mu;a,b}(x/y)\wt_{a,b}(\lambda/\mu; x,y)$, where 
	\[
		s_{\mu/\mu;a,b}(x,y) = \left(\frac{1-b_{-1}x}{1+b_{-1}y}\right)\left(\frac{1+b_{1}y}{1+ b_{1}y}\right),
	\]
	{\footnotesize
	\begin{align*}
		\wt_{\lambda/\mu;a,b}(x,y) = \frac{(1-a_{-2}b_{-2})(x+y)}{(1+b_{-2}y)(1-b_4 x)}\left(\frac{x-a_{-1}}{1-b_{-1}x}\right)\left(\frac{y+a_0}{1+b_0 y}\right)\left(\frac{y+a_1}{1+b_1 y}\right) \left(\frac{x-a_2}{1-b_2 x}\right)\left(\frac{x-a_3}{1-b_3 x}\right).
	\end{align*}}
	
	When $b = 0$, then the branching weight matches the example in Section 4 of \cite{OlRV03}:
	\[
		f_{\lambda/\mu;a,0}(x,y) = (x+y)(x-a_{-1})(y+a_0)(y+a_1)(x-a_2)(x-a_3).
	\]
\end{example}

The case of a hook $(p | q) = (p+1, 1^q)$ is particularly simple:
\begin{example}
    Let $x,y$ be single variables, then
    \[
    	s_{(p|q); a,b}(x/y) = \left(\frac{1-a_{q+1}'b_{q+1}'}{1-b_{q+1}' y}\right)\left(\frac{x+y}{1-b_{p+1}x}\right)(x | a,b)^p (y|a,b)^q.
    \]
\end{example}

To give formulas for the free fermionic Schur functions in many variables, we use the branching rules which follow immediately from the very definition. 

\begin{lemma}[Branching rules]\label{lem:branching}
	\[
        s_{\lambda/\mu; a,b}(x/y) = \sum_{\mu \subseteq \nu \subseteq \lambda}s_{\nu/\mu; a,b}(x_1,\dots,x_{n-1}/y_1,\dots,y_{n-1})s_{\lambda/\nu; a,b}(x_n/y_n).
    \]
    Note that the branching weights $s_{\lambda/\nu;a,b}(x_n/y_n)$ are given explicitly by \eqref{eq:onerow}. 
\end{lemma}

By repeating application of the branching rules, we get a combinatorial formula for the free fermionic Schur functions. In particular, the branching weights completely determine the free fermionic Schur functions. Since the branching weights are given in terms of ribbons, our result generalizes (4.5) in \cite{OlRV03} and (3.8) in \cite{Mol09}. We note that this formula can be equivalently rewritten in terms of diagonal-strict tableaux like in \cite{OlRV03}.

\begin{corollary}[Combinatorial ribbon formula]
    The following relations hold:
    \[
        s_{\lambda/\mu; a,b}(x/y) = \sum_{\mu = \nu_0 \subseteq \nu_1 \subseteq \dots \subseteq \nu_n = \lambda}s_{\nu_1/\nu_0;a,b}(x_1/y_1)\dots s_{\nu_n/\nu_{n-1};a,b}(x_n/y_n),
    \]
    where the skew diagram $\nu_k/\nu_{k-1}$ contains no $2 \times 2$ block of squares for each $k=1,2,\dots,n$.
\end{corollary}

Now we show that the free fermionic Schur functions generalize and unify the existing Schur functions from literature. We refer to \cite{Mac92,Mac95,OlRV03,Mol09} for the definitions of the corresponding functions. We also remark the the free fermionic Schur functions are related to $G_{\lambda/\mu}$ functions from \cite{ABPW21} by reparametrization of the weights.

\begin{theorem}[Degenerations]\label{thm:degenerations}
	The free fermionic Schur functions degenerate to the following functions from literature:
    \begin{enumerate}
        \item $s_{\lambda/\mu; 0,0}(x/0) = s_{\lambda/\mu}(x)$: classical Schur functions (see, e.g. \cite{Mac95}),
        \item $s_{\lambda/\mu;0,0}(x/y) = s_{\lambda/\mu}(x/y)$: supersymmetric Schur functions \cite{BR87},
        \item $s_{\lambda/\mu;a,0}(x/a) = s_{\lambda/\mu}(x \,||\, a)$: factorial Schur functions \cite{BL89, CL93, Mac92},
        \item $s_{\lambda/\mu;a,0}(x/y) = s_{\lambda/\mu}(x/y \,||\, a')$: factorial supersymmetric Schur functions \cite{Mol98},
        \item $s_{\lambda/\mu; a,0}(x/y) = s_{\lambda/\mu; a}(x,y)$: Frobenius-Schur functions \cite{OlRV03},
        \item $s_{\lambda/\mu;0,b}(x/0) = s_{\mu/\mu;b'}(x/0)\widehat{s}_{\lambda/\mu}(x \,||\, -b')$: dual Schur functions \cite{Mol09}.
    \end{enumerate}
	In all degenerations, $s_{\lambda/\mu;a,b}(x/a)$ stands for $s_{\lambda/\mu;a,}(x/y)$ with $y_i = a_i$ for $i = 1,2,3\dots$.
\end{theorem}
\begin{proof}
	Since the branching rules uniquely determine the Schur functions, it is enough to show that the branching weights $s_{\lambda/\mu;a,b}(x/y)$ for single variables $x,y$ degenerate to the branching weights of the corresponding specialization. Since the Frobenius-Schur functions generalize (1)--(4), it is sufficient to prove the statement only for the Frobenius-Schur functions and the dual Schur functions.
	
	For the Frobenius-Schur functions, let $b = 0$. Then $s_{\lambda/\mu;a,b}(x/y)$ degenerates to 
	\[
		s_{\lambda/\mu;a,0}(x/y) = \prod_{r \in R(\lambda/\mu)}(x+y)\prod_{\alpha \in h}\begin{cases}
			x-a_{c(\alpha)}, \quad &\text{if $\alpha$ has the left neighbor,}\\
			y+a_{c(\alpha)} \quad &\text{if $\alpha$ has the bottom neighbor}.
		\end{cases}
	\]
	It is precisely the branching weight $s_{\lambda/\mu; a} = f_{\lambda/\mu;a}$ for the Frobenius-Schur functions from Proposition 4.3 in \cite{OlRV03}.
	
	For the dual Schur functions, let $a = 0$ and $y = 0$. Then $s_{\lambda/\mu;a,b}(x/y)$ degenerates to 
	\begin{align*}
		s_{\lambda/\mu;0,a}(x/0) &= s_{\mu/\mu;0,b}(x/0)\prod_{r \in R(\lambda/\mu)}\frac{x}{1-b_{K+1}x}\prod_{\alpha \in r}\begin{cases}
			\frac{x}{1-b_{c(\alpha)}x}, \quad &\text{if $\alpha$ has the left neighbor,}\\
			0 \quad &\text{if $\alpha$ has the bottom neighbor}.
		\end{cases}.
	\end{align*}
	In other words, $s_{\lambda/\mu;0,a}(x/0) = 0$ unless $\lambda/\mu$ is a union of horizontal strips (i.e. every column of the diagram contains at most one box), in which case, 
	\[
		s_{\lambda/\mu;0,a}(x/0) = s_{\mu/\mu;0,b}(x/0)\prod_{\alpha \in \lambda/\mu}\frac{x}{1-b_{c(\alpha)+1}x},
	\]
	Finally, we note that when $\lambda/\mu$ is a horizontal strip, we have
	\[
		\prod_{\alpha \in \lambda/\mu}\frac{1-b_{c(\alpha)}x}{1-b_{c(\alpha)+1}x} = \frac{1-b_k x}{1-b_{K+1} x},
	\]
	where $k = \min_{\alpha \in \lambda/\mu}(c(\alpha))$ and $K = \max_{\alpha \in \lambda/\mu}(c(\alpha))$. Hence, we match the branching weight from the proof of Theorem 3.4 in \cite{Mol09} (after renaming $a \leftrightarrow b$ and applying the involution $b_i \leftrightarrow b_{-i+1}$ which is the same as $b \leftrightarrow -b'$).
\end{proof}

We remark that the free fermionic Schur functions simultaneously generalize the double Schur functions $s_{\lambda}(x \,||\, a)$ and the dual Schur functions $\widehat{s}_{\lambda}(x\,||\,a)$ thus providing a uniform framework for both of these families. Also, even the special case $s_{\lambda/\mu;0,a}(x/y)$ is already new and gives the supersymmetric version of the dual Schur functions by $s_{\lambda/\mu;0,a}(x/y)$, thereby extending the approach developed in \cite{Mol09}.

\subsection{Properties of Schur functions}
Now we show that the free fermionic Schur functions enjoy many familiar properties of the classical Schur functions. We note our approach gives novel proofs even for the known specializations. 

A function $f(x_1,\dots,x_n/y_1,\dots,y_n)$ is said to be supersymmetric if it is symmetric in variables $x$ and $y$, and satisfies the cancellation property: if $x_n = t$ and $y_n = -t$, then the function does not depend on $t$: 
\[
    f(x_1,\dots,x_{n-1},t/ y_1,\dots,y_{n-1},-t) = f(x_1,\dots,x_{n-1}/ y_1,\dots,y_{n-1}).
\]

\begin{proposition}\label{lem:supersymmetry}
    The free fermionic Schur functions $s_{\lambda/\mu;a,b}(x/y)$ are supersymmetric.
\end{proposition}
\begin{proof}
    The cancellation property and the separate symmetry follow from the first two properties of \Cref{lem:operatorrelationsforinfa}. 
\end{proof}

Thanks to the cancellation property, we can define the free fermionic Schur functions in $\infty + \infty$ variables $x=(x_1,x_2,\dots)$ and $y = (y_1,y_2,\dots)$ such that if $y_k = -x_k$ for all $k > n$, then
\begin{equation*}
    s_{\lambda/\mu;a,b}(x/y) = s_{\lambda/\mu;a,b}(x_1,\dots,x_n/ y_1,\dots,y_n).
\end{equation*}
More rigorously, the functions are the elements of the ring of the supersymmetric functions which is an inverse limit of the ring of graded rings. We refer to \cite{OlRV03, Mol09} for details. From now on we will not specify the number of variables unless when necessary. 

Now we show that the free fermionic Schur functions satisfy the determinant expressions that allow to express the general case $\lambda/\mu$ in terms of simpler shapes. Because of the normalization of the row transfer operators, in order to state the determinant identities, we need the following normalized form of the free fermionic Schur functions: 
\[
	\widetilde{s}_{\lambda/\mu;a,b}(x/y) = \prod_{i=1}^{n}\prod_{j={1}}^d \frac{1-b_{-j+1}x_i}{1+b_{-j+1}y_i}s_{\lambda/\mu;a,b}(x/y).
\]

\begin{theorem}[Determinant identities]
	The normalized functions $\widetilde{s}_{\lambda/\mu; a,b}$ satisfy the following determinant identities:
    \begin{enumerate}
        \item (Jacobi-Trudi formula)
        \[
            \widetilde{s}_{\lambda/\mu;a,b} = \det(\tau^{\mu_j-j+1}\widetilde{s}_{\lambda_i-\mu_j-i+j;a,b}).
        \]
        \item (N\"agelsbach-Kostka formula)
        \[
            \widetilde{s}_{\lambda/\mu;a,b} = \det(\tau^{-\mu_j'+j-1}\widetilde{s}_{(1^{\lambda_i-\mu_j-i+j});a,b}).
        \]
        \item (Giambelli formula)
        \[
            s_{\lambda;a,b} = \det(s_{(\alpha_i | \beta_j);a,b})_{1 \leq i,j \leq d(\lambda)}.
        \]
        \item (Ribbon formula)
        \[
            s_{\lambda;a,b} = \det(s_{[\alpha_i | \beta_j]; a,b})_{1 \leq i,j \leq r(\lambda)}.
        \]
    \end{enumerate}
\end{theorem}
\begin{proof}
	The proof of the Jacobi-Trudi identity follows from the LGV lemma for the free fermionic six vertex model (\Cref{lem:svlgv}). Indeed, for fixed $\lambda,\mu$, the free fermionic Schur functions are given by the finite six vertex model with $d = \ell(\lambda) = \ell(\mu)$ paths and $M$ columns, where $M$ is large enough. The paths enter the grid from the top at positions $(\mu_1, \mu_2 - 1, \dots, \mu_d - d + 1)$ and leave the grid at the bottom at positions $(\lambda_1, \lambda_2 - 1, \dots, \lambda_d - d + 1)$. Then by the LGV lemma for the six vertex model, we have
	\[
		\widetilde{s}_{\lambda/\mu;a,b}(x/y) = \prod_{i=1}^{n}\prod_{j={1}}^d \frac{1-b_{-j+1}x_i}{1+b_{-j+1}y_i}  s_{\lambda/\mu;a,b}(x/y) = \det\left(Z_{(\mu_i-i+1) \to (\lambda_j -j +1)}\right)_{1 \leq i,j \leq d}
	\]
	in the obvious notation, where paths enter from the top and leave at the bottom at the specified positions. By inspection, we have $Z_{(\mu_i-i+1) \to (\lambda_j-j+1)} = \tau^{\mu_i-i+1}Z_{(0) \to (\lambda_j - j - \mu_i + i)}$, and 
	\[
		Z_{(0) \to (k)} = \prod_{i=1}^{n}\frac{1-b_0x_i}{1+b_0y_i}s_{(k); a,b}(x/y) = \widetilde{s}_{(k); a,b}(x/y).
	\]
	
	Then the functions $\widetilde{s}_{\lambda/\mu; a,b}$ can be identified with a special case of the ninth variation introduced by Macdonald in \cite{Mac92}. Thus, the rest of the equations are the formal consequences. These are formulas (9.6), (9.6'), (9.7), and (9.9) in \cite{Mac92}. We only note that in the case of the Giambelli's formula and the Ribbon's formula, the normalization factors cancel out and we get a formula for the free fermionic Schur functions themselves. 
\end{proof}

In the classical (non-factorial) case, the supersymmetric Schur functions $s_\lambda(x/y)$ specialize to the Schur functions $s_\lambda(x)$ which have an explicit expression as a ratio of two determinants. This formula generalizes the determinant formula for the factorial Schur functions \cite{Mac92} and for the dual Schur functions from \cite{Mol09} at the same time. The same formula also appeared in \cite{MM11}.

For $\alpha = (\alpha_1,\dots,\alpha_n)$, let 
\begin{equation}
	A_{\alpha;a,b}(x_1,\dots,x_n) = \det\left(\frac{(x_j | a,b)^{\alpha_j}}{1-b_{\alpha_j+1}x_i}\right)_{1 \leq i,j \leq n}
\end{equation}

\begin{proposition}[Vandermonde determinant]
	Let $\rho_n = (n-1,n-2,\dots,0)$. Then
	\begin{equation}
		A_{\rho_n; a,b}(x_1,\dots,x_n) = \frac{\prod_{1 \leq i < j \leq n}(1-a_ib_j)(x_i-x_j)}{\prod_{i,j=1}^{n}(1-b_j x_i)}.
	\end{equation}
\end{proposition}
\begin{proof}
	By \Cref{lem:vandermondetype}.
\end{proof}

\begin{theorem}[Weyl determinant formula]
	\begin{equation}
		s_{\lambda;a,b}(x/a') = \frac{A_{\lambda+\rho_n;\tau^{-n}a,\tau^{-n}b}(x_1,\dots,x_n)}{A_{\rho_n;\tau^{-n}a,\tau^{-n}b}(x_1,\dots,x_n)},
	\end{equation}
	where $s_{\lambda;a,b}(x/a')$ stands for $s_{\lambda;a,b}(x/y)$ with $y_i = a_i' = -a_{-i+1}$ for $1 \leq i \leq n$. 
\end{theorem}
\begin{proof}
	When $y_k = a_k'$, by inspection, the partition function for $s_{\lambda;a,b}(x/y)$ matches
	\[
		\frac{Z_{\alpha;\tau^{-n}a,\tau^{-n}b}(x;a')}{\prod_{1 \leq i<j \leq n}(x_i-a_{-j+1})}.
	\]
	Now the claim follows from \Cref{lem:determinantfortype2}.
\end{proof}

Since these specializations are given as a ratio of two determinants, it is possible to prove the dual Cauchy identity by considering $A_{\rho_{2n}; a,b}(x,y; z,w)$. Since the derivation of this identity from the determinant is a routine, we omit the calculation. We only note that the interpretation of the free fermionic Schur functions as the partition functions of the six vertex model allows to give a combinatorial proof following the method from Section 8 in \cite{BMN14}.

Using the property that $\mathcal{A}_{a,b}(t/-t) = \operatorname{Id}$, and the separate symmetry in the first and second arguments, we conclude that
	\[
		\mathcal{A}_{a,b}(x,y) = \mathcal{A}_{a,b}(x,-t)\mathcal{A}_{a,b}(t,y)
	\]
	Let $t = (t_1,t_2,\dots,t_n)$, then by repeating application and separate symmetries, we have
	\[
		\mathcal{A}_{a,b}(x_1,y_1)\dots \mathcal{A}_{a,b}(x_n,y_n) = \mathcal{A}_{a,b}(x_1,-t_1)\dots \mathcal{A}_{a,b}(x_n,-t_n)\mathcal{A}_{a,b}(t_1,y_1)\dots \mathcal{A}_{a,b}(t_n,y_n).
	\]
	In terms of the six vertex model, it means that we replace $n$ rows labeled by $(x_k,y_k)$ with $2n$ rows, where first $n$ rows are labeled by $(x_k,-t_k)$, and the second $n$ rows are labeled by $(t_k,y_k)$. Since there is a bijection between the admissible states of the six vertex model and the semistandard tableaux, this factorization explains two different combinatorial formulae for the supersymmetric Schur functions: one in terms of diagonal-strict tableau with entries $\{1 < 2 < \dots < n\}$, and another in terms of $\mathbb{A}$-tableau with entries $\{1' < 1 < \dots < n < n'\}$ and $\mathbb{A}'$-tableaux with entries $\{1 < 1' < \dots < n' < n\}$. See Section 4 of \cite{OlRV03} and Section 1 of \cite{Mol98} and Section 2 of \cite{Mol09} and discussion therein for details. 
	
As a result, we get the following identity which generalizes the definition of the factorial supersymmetric Schur functions as given in \cite{Mol98,Mol09}. 

\begin{proposition}
    Let $t = (t_1,\dots,t_n)$ be a set of independent indeterminates. Then
    \begin{align*}
        s_{\lambda/\mu;a,b}(x/y) &= \sum_{\mu \subseteq \nu \subseteq \lambda}s_{\nu/\mu;a,b}(x/-t)s_{\lambda/\mu;a,b}(t/y).
    \end{align*}
\end{proposition}
\begin{proof}
	It follows from the branching rules and symmetries.
\end{proof}

By translating the results from the partition functions to the language of the free fermionic Schur functions, one can prove the Berele Regev factorization using \Cref{lem:bereleregevf}. However, we omit this result and only write the interesting special case. 
\begin{proposition}
	Let $x = (x_1,\dots,x_n)$ and $y = (y_1,\dots,y_n)$. Then
	\begin{align}
		\widetilde{s}_{[n^n];a,b}(x/y) = \prod_{1 \leq i,j \leq n}(x_i+y_j)\prod_{i=1}^{n}(1-a_i'b_i')\prod_{1 \leq i < j \leq n}(1-a_ib_j)(1-a_i'b_j').
	\end{align}
\end{proposition}
\begin{proof}
	It follows from \Cref{lem:boxvalue}.
\end{proof}

\subsection{Cauchy identity and corollaries}\label{sec:cauchy}
Let $x = (x_1,\dots,x_n)$, $y = (y_1,\dots,y_n)$, $a = (a_i)_{i \in \mathbb{Z}}$, and $b = (b_i)_{i \in \mathbb{Z}}$ be the same as at the beginning of the section. Let $\lambda$ and $\mu$ be two partitions. We define the \emph{dual free fermionic Schur functions} $\widehat{s}_{\lambda/\mu; a,b}(x/y)$ by
\begin{equation}
	\widehat{s}_{\lambda/\mu; a,b}(x/y) =  \braket{\lambda|\widehat{\mathcal{A}}_{a,b}(x_1,y_1) \dots \widehat{\mathcal{A}}_{a,b}(x_n,y_n)|\mu},
\end{equation}
where $\widehat{\mathcal{A}}_{a,b}(x,y)$ are the operators defined in \Cref{sec:infinitesv}. We use the same conventions in notation as we used for the free fermionic Schur functions $s_{\lambda/\mu; a,b}$. 

\begin{theorem}[Duality]\label{lem:duality}
    We have the following duality:
    \begin{equation}
        \widehat{s}_{\lambda/\mu;a,b}(x/y) = s_{\lambda'/\mu';b',a'}(y/x).
    \end{equation}
\end{theorem}
\begin{proof}
	It follows from the duality for the branching weights in \Cref{lem:weightduality}.
\end{proof}

As a result, the dual free fermionic Schur functions are also supersymmetric and satisfy the same properties as we described in the previous section. 

The hidden symmetry of the parameters resembles the duality between parameters $q$ and $t$ in Macdonald polynomials. In particular, Macdonald proves in (5.1) of \cite{Mac95} that $\omega_{q,t} P_\lambda(x; q,t) = Q_\lambda(x; t, q)$, where $\omega_{q,t}$ is the appropriate involution. We refer to \cite{Mac95} for details. 

Now we are ready to prove the Cauchy identity for the free fermionic Schur fuctions and their duals. These results should be understood in the sense of discussion at the end of \Cref{sec:functionalrelations}.

The most general form of the Cauchy identity is given by the supersymmetric skew Cauchy identity. This result generalizes the skew Cauchy identity from Proposition 3.7. from \cite{ABPW21}.

\begin{proposition}[Skew Cauchy Identity]\label{lem:skewcauchy}
    We have the following identity:
    \[
        \sum_{\lambda}s_{\lambda/\mu; a,b}(x/y)\widehat{s}_{\lambda/\nu; a,b}(z/w) = \prod_{i,j}\frac{1+y_iz_j}{1-x_iz_j}\frac{1+x_iw_j}{1-y_iw_j}\sum_{\rho}\widehat{s}_{\mu/\rho; a,b}(z/w)s_{\nu/\rho; a,b}(x/y).
    \]
\end{proposition}
\begin{proof}
	By repeating application of the operator Cauchy identity from \Cref{thm:operatorcauchy}, we can write 
	\[
		\langle \mu | \mathcal{A}_{a,b}(x_1,y_1)\dots \mathcal{A}_{a,b}(x_n,y_n)\widehat{\mathcal{A}}_{a,b}(z_1,w_1)\dots \widehat{\mathcal{A}}_{a,b}(z_n,w_n) | \nu \rangle
	\]
	as 
	\[
		\prod_{i,j}\frac{1+y_iz_j}{1-x_iz_j}\frac{1+x_iw_j}{1-y_iw_j}\langle \mu | \widehat{\mathcal{A}}_{a,b}(z_1,w_1)\dots \widehat{\mathcal{A}}_{a,b}(z_n,w_n)\mathcal{A}_{a,b}(x_1,y_1)\dots \mathcal{A}_{a,b}(x_n,y_n) | \nu \rangle.
	\]
	Now express both sides in terms of the free fermionic Schur functions. 
\end{proof}

Now we discuss multiple corollaries of the skew Cauchy identity and their relations to the results from literature for the specializations. 

The following special case generalizes Theorem 3.4 in \cite{Mol09}.

\begin{corollary}
    \begin{align*}
        \prod_{i,j}\frac{1+y_iz_j}{1-x_iz_j}\frac{1+x_iw_j}{1-y_iw_j}s_{\nu; a,b}(x/y) &= \sum_{\nu \subset \lambda}s_{\lambda; a,b}(x/ y)\widehat{s}_{\lambda/\nu; a,b}(z/w),\\
        \prod_{i,j}\frac{1+y_iz_j}{1-x_iz_j}\frac{1+x_iw_j}{1-y_iw_j}\widehat{s}_{\mu; a,b}(z/w) &= \sum_{\mu \subset \lambda}\widehat{s}_{\lambda; a,b}(z/w)s_{\lambda/\mu; a,b}(x/y).
    \end{align*}
\end{corollary}
\begin{proof}
	The first and the second identity correspond to $\mu = \emptyset$ and $\nu = \emptyset$ in \Cref{lem:skewcauchy}, respectively.
\end{proof}

We now prove our main result---the supersymmetric Cauchy identity in the form of Berele-Regev \cite{BR87}. We remark that the right-hand side of the identity is independent of the parameters $(a_i)_{i \in \mathbb{Z}}$ and $(b_i)_{i \in \mathbb{Z}}$. This identity degenerates to Theorem 3.1 and Corollary 3.2 from \cite{Mol09}. It also generalizes Theorem 3.8 from \cite{ABPW21}.

\begin{theorem}[Cauchy Identity]\label{thm:cauchy}
    We have the following identity:
    \[
        \sum_{\lambda}s_{\lambda; a,b}(x/y)\widehat{s}_{\lambda; a,b}(z/w) = \prod_{i,j}\frac{1+y_iz_j}{1-x_iz_j}\frac{1+x_iw_j}{1-y_iw_j}.
    \]
\end{theorem}
\begin{proof}
    Setting $\mu = \nu = \emptyset$ in Proposition~\ref{lem:skewcauchy} yields the desired identity.
\end{proof}

As an application, we give the generating series for the hook Schur functions $s_{(p|q);a,b}$. This result generalizes Proposition 7.1 from \cite{OlRV03}. 

\begin{corollary}[Generating series for hook functions]\label{lem:genseriesforhooks}
    \begin{align*}
        1 + (z+w)\sum_{p,q=0}^{\infty}s_{(p | q); a,b}(x/y)(1-a_{p+1}b_{p+1})\frac{(z|b)^p}{(z;a)^{p+1}}\frac{(w|b')^q}{(w;a')^{q+1}} &= \prod_{i}\frac{1+y_iz}{1-x_iz}\frac{1+x_iw}{1-y_iw},\\
        1 + (z+w)\sum_{p,q=0}^{\infty}\widehat{s}_{(p | q); a,b}(x/y)(1-a_{q+1}'b_{q+1}')\frac{(z|a)^p}{(z;b)^{p+1}}\frac{(w|a')^{q}}{(w; b')^{q+1}} &= \prod_{i}\frac{1+x_i w}{1-x_i z}\frac{1+y_i z}{1-y_i w}.
    \end{align*}
\end{corollary}
\begin{proof}
    Let $z,w$ be single variables in the Cauchy identity. Then $\widehat{s}_{\lambda;a,b}(z/w)$ is zero unless $\lambda = (p|q)$ is a hook, in which case, $\widehat{s}_{(p|q);a,b}(z/w)$, which is given explicitly. The second identity follows from the duality.
\end{proof}

The generating series allows us to compute the hook functions iteratively, mimicking the divided difference algorithm for the Newton interpolation problem. 

\begin{example}\label{ex:p0q0}
	Set $z = b_1$ and $w = b_1' = -b_0$ in the generating series. Then we get
	\[
		1+(b_1-b_0)\frac{s_{(0|0);a,b}(x/y)}{1-a_0b_0} = \prod_{i}\frac{1+y_i b_1}{1-x_i b_1}\frac{1-x_i b_0}{1+y_i b_0},
	\]
	or 
	\[
		s_{(1);a,b}(x/y) = s_{(0|0);a,b}(x/y) = \frac{1-a_0b_0}{b_1-b_0}\left(\prod_{i}\frac{1-y_i b_1}{1+x_i b_1}\frac{1-x_i b_0}{1+y_i b_0} - 1\right).
	\]
	We remark that without the second set of parameters $b = (b_i)_{i \in \Z}$, it would not be possible to give the explicit formula for this function. Moreover, the formula makes sense when $b = 0$ as the limit of the right side as $b_0,b_1 \to 0$.
\end{example}

By specializing further, we get generating series for the complete homogeneous functions and the elementary functions (and their duals). 

Let $h_{p+1;a,b} = s_{(p|0);a,b} = s_{(p+1);a,b}$ and $\widehat{h}_{p+1;a,b} = \widehat{s}_{(p|0);a,b} = \widehat{s}_{(p+1);a,b}$ be the complete symmetric functions, and $e_{q+1;a,b} = s_{(0|q);a,b} = s_{(1^{q+1});a,b}$ and $\widehat{e}_{q+1;a,b} = \widehat{s}_{(0|q);a,b} = \widehat{s}_{(1^{q+1});a,b}$ be the elementary symmetric functions. We define $h_{0;a,b} = \widehat{h}_{0;a,b} = e_{0;a,b} = \widehat{e}_{0;a,b} = 1$.

Let $\tau^r$ be an operator that acts on sequences $(a_i)_{i \in \mathbb{Z}}$ by shifting the indices by $\tau^r(a_i) = (a_{i+r})_{i \in \mathbb{Z}}$. We also write $\tau^r s_{\lambda/\mu;a,b}$ for $s_{\lambda/\mu; \tau^r a, \tau^r b}$.

\begin{corollary}[Generating series for complete and elementary functions]\label{lem:genseriesforhande}
    \begin{align}
        1 + \sum_{k=1}^{\infty}h_{k;a,b}(x,y)\frac{1-a_kb_k}{1-a_0b_0}\frac{(z|\tau^{-1} b)^{k}}{(z;a)^{k}} &= \prod_{i}\frac{1+y_iz}{1-x_iz}\frac{1-b_0 x_i}{1 + b_0y_i},\\
        1 + \sum_{k=1}^{\infty}e_{k; a,b}(x,y)\frac{1-a_k b_k}{1-a_1b_1}\frac{(w | (\tau^{-1} b)')^k}{(w ; a')^k} &= \prod_{i}\frac{1+x_iw}{1-y_iw}\frac{1+y_i b_1}{1-x_i b_1},\\
        1 + \sum_{k=0}^{\infty}\widehat{h}_{k; a,b}(z,w)\frac{1-a_{k'}b_{k'}}{1-a_0b_0}\frac{(x | \tau^{-1}a)^k}{(x;b)^{k}} &= \prod_{i}\frac{1-z_ja_0}{1-z_jx}\frac{1+w_j x}{1+a_0w_j},\\
        1 + \sum_{k=0}^{\infty}\widehat{e}_{k;a,b}(z,w)\frac{1-a_{k'}b_{k'}}{1-a_1b_1}\frac{(y|(\tau^{-1}a)')^k}{(y;b')^k} &= \prod_{i}\frac{1+z_jy}{1-w_jy}\frac{1+a_1w_j}{1-a_1z_j}.
    \end{align}
\end{corollary}
\begin{proof}
    We prove the first identity by setting $w = -b_0$ in the generating series for complete homogeneous symmetric functions, as given by the first identity in \Cref{lem:genseriesforhooks}. Using the property $(b_0^{-1}; b')^q = 0$ for $q > 0$, we can simplify the sum to only include $p = 0,1,2,\dots$. 
    
    Similarly, we can obtain the other identities by setting $z = b_1$, $y = -a_0$, and $x = a_1$ in the generating series for elementary homogeneous symmetric functions and complete homogeneous symmetric functions, respectively. Then, we simplify the expressions to obtain the desired identities.
\end{proof}

\appendix

\section{The LGV lemma for the six vertex model}\label{sec:lgvforsv}
We briefly revisit the theory of non-intersecting lattice paths and recall the powerful Lindstr\"om-Gessel-Viennot lemma (the LGV lemma). For a comprehensive treatment of the topic, we refer the reader to \cite{L73,GV85}.

Consider a directed acyclic graph $G$ (a graph with directed edges and with no cycles), in which each directed edge $e \in G$ is assigned a weight $\wt(e)$. For a directed path $P$ between two vertices, we define the weight of the path, $\wt(P)$, as the product of the weights of the edges in the path. For any two vertices $a,b \in G$, we define the sum $e(a,b) = \sum_{P\colon a \to b}\wt(P)$ over all directed paths from $a$ to $b$.

Let $A = (a_1,\dots,a_n)$ and $B = (b_1,\dots,b_n)$ be two $n$-tuples of vertices. We consider an $n$-tuple of non-intersecting paths $(P_1,\dots,P_n)\colon A \to B$, where $P_i\colon a_i \to b_i$. The weight $\wt(P_1,\dots,P_n)$ of the $n$-tuple is defined as the product $\wt(P_1,\dots,P_n) = \prod_{i=1}^{n}\wt(P_i)$ of the weights of the involved paths. Additionally, we impose the restriction that if we fix the starting points $(a_1,\dots,a_n)$, then each path $P_i$ in an $n$-tuple $(P_1,\dots,P_n)$ of non-intersecting paths must end exactly at $b_i$. In other words, there is no $n$-tuple of non-intersecting paths $P_1,\dots,P_n$ such that $P_i\colon a_i \to b_{\sigma(i)}$ for some non-identity permutation $\sigma \in S_n$.

With these conditions in place, we can state the Lindstr\"om-Gessel-Viennot theorem lemma:

\begin{lemma}[The LGV lemma]

The weighted sum of all $n$-tuples $(P_1,\dots,P_n)\colon A \to B$ is equal to a determinant involving only the weights of systems with one path:
\[
    \sum_{(P_1,\dots,P_n)\colon A \to B}\wt(P_1,\dots,P_n) = \det(e(a_i,b_j))_{1\leq i,j \leq n}.
\]
\end{lemma}

The LGV lemma is applicable in a vastly more general context, including the possibility of permutations of the paths and more general graphs. However, in this paper, we focus on the most basic and special case.

The six vertex model resembles a model of non-intersecting lattice paths, with the exception that the paths \textit{do} intersect. However, if the weights satisfy certain conditions, it is possible to adjust these intersections to obtain a model of non-intersecting paths without altering the normalized partition function.

We say that the six vertex model is \textit{free fermionic} if the weight functions satisfy the following condition:
\begin{equation}
    \sva_1\sva_2 + \svb_1\svb_2 = \svc_1\svc_2.
\end{equation}

The weights we are working with given by \eqref{eq:ffweights} are free fermionic because 
\begin{equation}
	(1-b x)(y+a) + (1+by)(x-a) = (1-ab)(x+y).
\end{equation}

When computing partition functions, by renormalizing the weights, we can assume that the weights $\sva_1 = 1$ for all vertices. Then the normalized partition function differs from the original partition function only by the product of $\sva_1(v)$ over all vertices $v$ in the model. 

Now, to each six vertex model, we want to associate a system of non-intersecting lattice paths. We do it as follows. We construct the directed graph on top of the grid of the corresponding six vertex model. We place vertices at all positions of vertices of the original grid, and extra vertices in the middle of all edges appearing in the grid for the six vertex model. All original horizontal edges are directed from left to right, and all original vertical edges are directed from top to bottom. Additionally, we add a new diagonal edge connecting the midpoint of the top edge to the midpoint of the right edge and directed from northwest to southeast. See \Cref{fig:svlgvgraph} for the illustration. Then we assign weights to all directed edges in \Cref{fig:nonintersectingweights}, where we give the weight for the highlighted edge. 

\usetikzlibrary{arrows}
\tikzset{edge/.style = {->,> = latex'}}

\begin{figure}
    \centering

    \begin{tikzpicture}[xscale=0.7, yscale=0.7]
    	
        \foreach \x in {1,3,5,7} {
            \draw[edge] (0,\x) -- (16,\x);
            \pgfmathtruncatemacro{\index}{5-(\x+1)/2}
            \node [left] at (0,\x) {$x_\index, y_\index$};
        }
        \foreach \y in {1,3,5,7,9,11,13,15} {
            \draw (\y,0) -- (\y,8);
            \pgfmathtruncatemacro{\index}{(\y+1)/2}
            \node [above] at (\y,8) {$a_\index, b_\index$};
        }
        
        \foreach \x in {1,3,5,7} {
        	\foreach \y in {1,3,5,7,9,11,13,15} {
        		\node[draw, circle, fill=black, inner sep=1pt] at (\y, \x) {};
        		\draw[edge] (\y, \x + 1) -- (\y + 1, \x);
        	}
        }
        
        \foreach \x in {1,3,5,7} {
        	\foreach \y in {1,...,17} {
        		\node[draw, circle, fill=black, inner sep=1pt] at (\y-1, \x) {};
        	}
        }
        
        \foreach \x in {1,3,5,7} {
        	\foreach \y in {1,...,16} {
        		\draw[edge] (\y - 1, \x) -- (\y, \x);
        	}
        }
        
        \foreach \x in {1,...,9} {
        	\foreach \y in {1,3,5,7,9,11,13,15} {
        		\node[draw, circle, fill=black, inner sep=1pt] at (\y, \x-1) {};
        	}
        }
        
        \foreach \x in {1,...,8} {
        	\foreach \y in {1,3,5,7,9,11,13,15} {
        		\draw[edge] (\y, \x) -- (\y, \x - 1);
        	}
        }

        \draw[path] (0,3) -- (1,3) -- (3,3) -- (3,1) -- (7,1) -- (7,0);
        
        \draw[path] (3,8) -- (3,7) -- (5,7) -- (5, 5) -- (7, 5) -- (7, 3) -- (11, 3) -- (11, 0);
        
        \draw[path] (5, 8) -- (6, 7) -- (7, 7) -- (7, 6) -- (8, 5) -- (13, 5) -- (13, 3) -- (15, 3) -- (15, 1) -- (16, 1);
    \end{tikzpicture}
    \caption{A typical state in the acyclic graph that corresponds to the six vertex model. Compare with \Cref{fig:svmodel}.}
    \label{fig:svlgvgraph}
\end{figure}

\begin{figure}[h]
    \centering
    \begin{center}
    \begin{tabular}{ c||c|c|c|c|c } 
    edge & 
    \begin{tikzpicture}[scale=0.8]
        \draw (0,0) to (2,0);
        \draw (1,-1) to (1,1);
        \draw (1,1) to (2, 0);
        \draw[path] (0,0) to (1,0);
    \end{tikzpicture} & 
    \begin{tikzpicture}[scale=0.8]
        \draw (0,0) to (2,0);
        \draw (1,-1) to (1,1);
        \draw (1,1) to (2, 0);
        \draw[path] (1,0) to (1,-1);
    \end{tikzpicture} &
    \begin{tikzpicture}[scale=0.8]
        \draw (0,0) to (2,0);
        \draw (1,-1) to (1,1);
        \draw (1,1) to (2, 0);
        \draw[path] (1,1) to (1,0);
    \end{tikzpicture} &
    \begin{tikzpicture}[scale=0.8]
        \draw (0,0) to (2,0);
        \draw (1,-1) to (1,1);
        \draw (1,1) to (2, 0);
        \draw[path] (1,0) to (2,0);
    \end{tikzpicture} &
    \begin{tikzpicture}[scale=0.8]
        \draw (0,0) to (2,0);
        \draw (1,-1) to (1,1);
        \draw (1,1) to (2, 0);
        \draw[path] (1,1) to (2,0);
    \end{tikzpicture}\\
    \hline
    weight & $\svc_1(v)$ & $1$ & $\svb_1(v)$ & $\svb_2(v)/\svc_1(v)$ & $\sva_2(v)/\svc_1(v)$
    \end{tabular}
    \end{center}
\caption{The weights of edges around the vertex $v$ in the associated graph.}\label{fig:nonintersectingweights}
\end{figure}
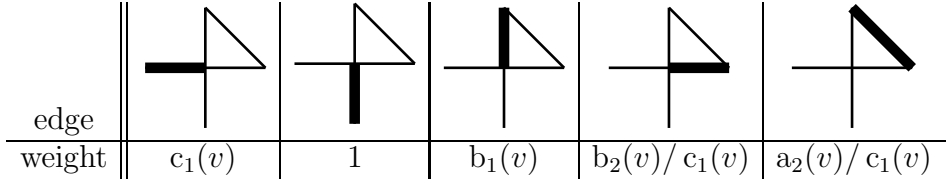

\begin{proposition}[The LGV lemma for the Six Vertex Model]\label{lem:svlgv}
	Consider a free fermionic six vertex model with normalized weights so that $\sva_1(v) = 1$ for all vertices $v$. Let $A_1,A_2,\dots,A_d$ be the positions where paths enter the model counting from the bottom left corner, to the top left corner, and then to the top right corner, and $B_1,B_2,\dots,B_d$ be the positions where paths leave the model, counting from the bottom left corner, to the bottom right corner, and then to the top right corner. The partition function $Z$ of a free fermionic six vertex model is given by the determinant of one-path partition functions:
    \[
        Z = \det\left(Z_{A_i \to B_j}\right),
    \]
    where $Z_{A_i \to B_j}$ is the normalized partition function of the system with one path entering at the position $A_i$ and leaving at the position $B_j$. 
    
    In the case, when the weights are not normalized, we have
	\[
	    \frac{Z}{\prod_{v}\sva_1(v)} = \det\left(\frac{Z_{A_i \to B_j}}{\prod_v \sva_1(v)}\right).
	\]
\end{proposition}
\begin{proof}

We show that the weighted sum over the non-intersecting lattice paths for the constructed graph is equal to the partition function of the original six vertex model. We demonstrate the equality at the local level of a single vertex. Then the result follows globally. It suffices to show that for each type of vertex in the six-vertex model the associated weights in the graph give the same contribution. The weights of types $\sva_1,\svb_1,\svb_2,\svc_1$ are just mapped to the same weights. Consider a vertex of type $\sva_2$ in the six vertex model which corresponds to the intersection of two paths. In the associated graph, the paths do not intersect:
    \[
        \vcenter
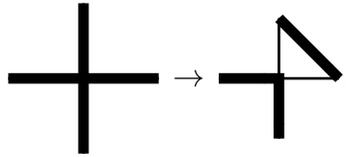
{\hbox{\begin{tikzpicture}
            \draw (0,0) -- (2,0);
            \draw (1, -1) -- (1, 1);
            \draw[path] (0,0) -- (2,0);
        	\draw[path] (1,-1) -- (1,1);
        \end{tikzpicture}}} \to
        \vcenter{\hbox{\begin{tikzpicture}[scale=0.8]
            \draw (0,0) to (2,0);
            \draw (1,-1) to (1,1);
            \draw (1,1) to (2, 0);
            \draw[path] (1,1) -- (2, 0);
            \draw[path] (0,0) -- (1, 0) -- (1, -1);
            \draw[edge] (0,0) -- (1,0);
            \draw[edge] (1,0) -- (1, -1);
            \draw[edge] (1,1) -- (2, 0);
        \end{tikzpicture}}}.
    \]
    
    The weight of the vertex of type $\sva_2$ splits into the product of two edge weights: $\svc_1$ and $\sva_2/\svc_1$. Hence, that the weight for vertices of type $\sva_2$ is preserved. Next consider a vertex of type $\svc_2$ in the six vertex model. In the associated graph, there are two possibilities for the path: 
    \[
        \vcenter
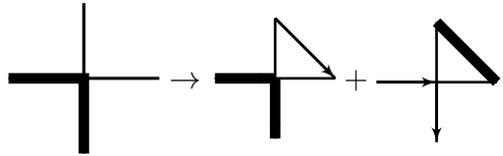
{\hbox{\begin{tikzpicture}
            \draw (0,0) -- (2,0);
            \draw (1, -1) -- (1, 1);
            \draw[path] (0,0) to (1,0) to (1,-1);
        	\draw (1,1) to (1,0) to (2,0);
        \end{tikzpicture}}} \to
        \vcenter{\hbox{\begin{tikzpicture}[scale=0.8]
            \draw (0,0) to (2,0);
            \draw (1,-1) to (1,1);
            \draw (1,1) to (2, 0);
            \draw[path] (0,0) to (1,0);
            \draw[path] (1,0) to (1,-1);
            \draw[edge] (0,0) -- (1,0);
            \draw[edge] (1,0) -- (1, -1);
            \draw[edge] (1,1) -- (2, 0);
        \end{tikzpicture}}} +  \vcenter{\hbox{\begin{tikzpicture}[scale=0.8]
            \draw (0,0) to (2,0);
            \draw (1,-1) to (1,1);
            \draw (1,1) to (2, 0);
            \draw[path] (1,1) to (2, 0);
            \draw[edge] (0,0) -- (1,0);
            \draw[edge] (1,0) -- (1, -1);
            \draw[edge] (1,1) -- (2, 0);
        \end{tikzpicture}}}.
    \]
    
    Then the six vertex weight $\svc_2$ splits into the sum of $\svb_1 \svb_2/\svc_1$ and $\sva_2/\svc_1$. But by the free-fermionic condition for the normalized weights, we have $c_1c_2 = b_1b_2 + a_2$. Hence, the weight for vertices of type $\svc_2$ is preserved. 	
\end{proof}

\bibliographystyle{alphaurl}
\bibliography{ffschur.bib}

\end{document}